\documentclass [12pt,a4paper]{article}
\usepackage{lastpage}
\usepackage{extramarks,enumerate}
\usepackage{graphicx}
\usepackage{caption}
\usepackage{lipsum}
\usepackage{amsmath,mathtools,amsthm}
\usepackage{subcaption}
\usepackage{amsmath}
\usepackage{amsthm}
\usepackage{lmodern}
\usepackage{tikz,color}
\usepackage{latexsym,amssymb}
\usepackage{amssymb}
\usepackage{enumerate,pgfplots}
\usepackage{hyperref}
\DeclareMathOperator{\supp}{supp}

\newtheorem{con}{Conjecture}
\newtheorem{theo}{Theorem}
\theoremstyle{definition}
\newtheorem{disc}{Discussion}
\newtheorem{obs}{Observation}

\theoremstyle{remark}
\newtheorem*{rem}{Remark}

\makeatletter
\renewcommand\p@enumi{\arabic{section}.}
\makeatother

\def\eps{\varepsilon}

\begin{document}

\title{Mean-field coupling of identical expanding circle maps}
\title{\textbf{Mean-field coupling of identical expanding circle maps}}
\author{\textbf{Fanni S\'elley and P\'eter B\'alint}
\medskip\\
MTA-BME Stochastics Research Group\\
Budapest University of Technology and Economics\\
Egry J\'ozsef u. 1, H-1111 Budapest, Hungary
\medskip \\and \medskip \\
 Department of Stochastics, Institute of Mathematics,\\
Budapest University of Technology and Economics\\
Egry J\'ozsef u. 1, H-1111 Budapest, Hungary\\
\texttt{selley@math.bme.hu, pet@math.bme.hu}\\
}
\date{\today}
\maketitle

\begin{abstract}
Globally coupled doubling maps are studied in this paper.
In this setting and for finitely many sites, two distinct bifurcation values
of the coupling strength have been identified in the literature, corresponding
to the emergence of contracting directions (\cite{koiller2010coupled}) and,
specifically for $N=3$ sites, to the loss of ergodicity (\cite{fernandez2014breaking}).
On the one hand, we reconsider these results and provide an interpretation of the observed dynamical phenomena
in terms of the synchronization of the sites. On the other hand, we initiate a new point of view which focuses on the
evolution of distributions and allows to incorporate the investigation of a continuum of sites. In particular,
we observe phenomena that is analogous to the limit states of the contracting regime of
$N=3$ sites.
\end{abstract}

\section{Introduction}

In this paper we study systems of coupled maps. In a broad sense, this means that we are given a network of (finitely or infinitely many) interacting particles, termed sites. It is convenient to think of the network as a (directed or undirected) graph, with a site located at each node, typically modeled by a discrete time dynamical system. The dynamics of the compound system consist of two components: in addition to the evolution of individual maps, the sites interact along the edges of the graph. The form of the interaction can be of several types -- later we restrict to one of the most natural choices, a diffusive coupling -- yet it is typically the case that there is some external parameter that measures the strength of the coupling.

From a statistical mechanics point of view, arguably the main interest in such models is the emergence of bifurcations: how do the characteristic features of such a compound system change when the interaction strength is varied. When there is no coupling, the individual sites behave independently; on the other hand for strong interactions some kind of synchronization can be expected. Such phenomena can be thought of as a deterministic analogue of the phase transitions of Ising models in statistical physics (\cite{bunimovich1988spacetime}, \cite{miller1993macroscopic}, \cite{gielis2000coupled}).

The topic has an extensive literature and we only cite papers with more direct relevance to our work, for more complete lists see the references in \cite{koiller2010coupled}, \cite{fernandez2014breaking}, \cite{keller2006uniqueness}  and the collection \cite{chazottes2005dynamics}.
A particularly popular scenario is that of \textit{coupled map lattices}, which corresponds to an \textit{infinite graph} (typically $\mathbb{Z}^d$) modeling the network of interaction. Mathematically rigorous results use most sophisticated tools of dynamical systems theory, and typically prove the lack of phase transitions (unique SRB measure) for small interaction strength (\cite{bunimovich1988spacetime}, \cite{bricmont1996high}, \cite{jarvenpaa1997srb}, \cite{jiang1998equilibrium}, \cite{fischer2000transfer}, \cite{keller2005spectral}, \cite{keller2006uniqueness}). Similar phenomenon is observable when the coupling strength is not necessarily low, but the interaction network is a sufficiently sparse graph. In this case the sparsity of the interaction network causes the weakness of interaction and a behaviour similar to that of the uncoupled system. For example, random networks with a few hubs and many low degree nodes were studied by Pereira, van Strien and Lamb \cite{pereira2013dynamics}. There are only a few mathematically rigorous results which prove the presence of a phase transition in the infinite system (the thermodynamic limit) for certain specific situations (\cite{bardet2009stochastically}, \cite{bardet2006phase}).

Of special interest to us is the case of \textit{globally (or mean field) coupled maps}. In this case, with quadratic or tent maps at the individual sites, unexpected dynamical behaviour (termed violation of the law of large numbers by Kaneko) has been observed (\cite{kaneko1990globally}, \cite{kaneko1995remarks}, \cite{ershov1995mean}, \cite{nakagawa1996dominant}). When the individual site maps are smooth expanding circle maps, such phenomena are absent (\cite{jarvenpaa1997srb}, \cite{keller2000ergodic}); on the other hand, \cite{bardet2009stochastically} proved the presence of a phase transition analogous to that
of the Curie-Weiss model for a special class of smooth (fractional linear) site maps, a bifurcation that occurs exclusively in the infinite system (the thermodynamic limit).

A parallel line of investigation, recently (re)initiated in \cite{koiller2010coupled}, is to consider \textit{finitely many sites}, however, not only in the weakly coupled case, but \textit{for a wider range of coupling parameters}. Such models can be regarded as dynamical systems acting on some finite, yet high dimensional manifold. As the individual site dynamics are expanding, we get a fully expanding system for weak coupling, but for stronger coupling strength, \textit{contracting directions emerge} as a sign of synchronization (\cite{just1995globally}, \cite{boldrighini1995ising}, \cite{boldrighini2001ising}). Young and Koiller \cite{koiller2010coupled} focused on how the geometrical and topological properties of the system behave at such transitions.

Later Fernandez \cite{fernandez2014breaking} extended this analysis to ergodic properties. He studied \textit{mean field diffusive coupling of doubling maps}, and observed \textit{another transition prior to the emergence of contracting directions}. By general arguments, we know that for small coupling strength there exists a unique ergodic absolutely continuous invariant measure. However, as the coupling strength is increased, Fernandez detected \textit{the appearance of multiple ergodic components}. More precisely, it is proved in \cite{fernandez2014breaking} that for \textit{$N=3$ sites}, if the parameter is raised above a certain value, there exist at least two ergodic components (distinct absolutely continuous invariant measures).\footnote{In fact, \cite{fernandez2014breaking} detected two distinct values of the coupling parameter, $\eps_1<\eps_2$, both within the expanding regime, and showed that there are at least two ergodic components for $\eps\in(\eps_1,\eps_2)$, and that there are at least six ergodic components for $\eps>\eps_2$. Here we show that there are at least six ergodic components for any $\eps>\eps_1$.} Furthermore, based on numerical simulations and geometric considerations, Fernandez argued that the emergence of multiple ergodic components corresponds to the breaking of certain symmetries, namely the inversion and the permutation symmetry of the site locations on the circle. It is important to point out that the observed breaking of ergodicity occurs already in the expanding regime of the coupling parameter. For contracting parameters, there are two possible limit behaviours for $N=3$ sites as already observed in \cite{koiller2010coupled}. Based on simulations Fernandez expected similar phenomena if $N\ge 4$, which, however, turns out to be too complicated for analytic considerations. Finally, he showed that there is no ergodicity breaking for $N=2$ sites.

In this paper we would like to address the following questions:
\begin{itemize}
\item Is there a way to identify the different limit behaviours, in particular, the ergodicity breaking observed in \cite{fernandez2014breaking} at the level of individual sites? In particular, is it possible to interpret these bifurcations as some synchronization phenomena, and if yes, in what sense?
\item How general are the observed phenomena as far as the number of sites is concerned, specifically, do they occur in the thermodynamic limit? Note that, as pointed out in \cite{fernandez2014breaking}, increasing the number of sites (in fact, already $N=4$) results in geometric complications, hence it may be more feasible to study directly a system that can be interpreted as a version of the model with infinitely many sites .
\end{itemize}

As we intend to follow up on the results of Fernandez, \textit{we restrict to mean field diffusive coupling of doubling maps} from now on. In accordance with the questions formulated above, we aim to investigate both the finite system, and the system of a continuum of sites. On the one hand, for \textit{finitely many sites} the results of \cite{fernandez2014breaking} are reconsidered and slightly extended. In the case of \textit{coupling two maps} we show that the ergodic invariant measure, even though unique for the whole expanding regime, \textit{ceases to be mixing} after the coupling strength is increased above some threshold value. In the case of \textit{coupling three maps}, we provide a \textit{better upper bound} for the threshold value of the coupling parameter where (at least) \textit{six ergodic components} emerge. Furthermore, the different ergodic components are \textit{identified at the level of the individual sites}. This identification provides further evidence that the loss of ergodicity corresponds to the breaking of the inversion and permutation symmetries. In particular, some of the ergodic components can be mapped onto each other by relabeling the sites, hence this phenomenon is not fully observable if indistinguishable sites are considered.

On the other hand, the model is reconsidered from a different point of view, instead of the individual sites, we investigate the \textit{evolution of the distribution} (the measure) they generate on the circle. Pure point measures -- when the distribution is an average of $N$ Dirac masses -- correspond precisely to the system of a finite number of coupled sites studied in\cite{fernandez2014breaking}. However, this viewpoint allows us to extend our investigation to \textit{absolutely continuous distributions}, a property that is preserved by the dynamics. This later setting will be referred to as a \textit{continuum of sites}. Our motivation is to investigate a system which can be regarded as an infinite-site-version of the model.
We believe that, for a suitable sequence of initial conditions, the behaviour of a continuum of sites could be obtained as a limit of the models with a finite number of sites; however, to avoid confusion, it is important to point out that we do not consider, and do not claim anything about this limiting process here. Nonetheless, simply by investigating the evolution of distributions, we may make some comparison of the finite and the continuum cases. We observe bifurcations in the system of a continuum of sites which correspond to the expanding-contracting transition, and accordingly, two limit behaviours are identified which are analogous to those observed already in \cite{koiller2010coupled} (and reconsidered in \cite{fernandez2014breaking}) within the contracting regime for $N=3$ sites. However, no phenomenon analogous to the breaking of ergodicity within the expanding regime is observed for a continuum of sites. As for the breaking of the permutation symmetry, this may be related to the fact that in the thermodynamic limit of a mean field model it does not seem natural to distinguish the sites.

The rest of this paper is organized as follows. In section \ref{2} we provide the setting and state our results.  These are formulated, on the one hand, as mathematical theorems, but on the other hand some discussions are also added to explain what type of synchronization phenomena is observed at the level of the sites. In section \ref{3} we deal with the case of finite system size: after some general observations we give a detailed analysis of systems with two and three interacting sites. In section \ref{4} we turn to the case of the system with a continuum of sites: first we study the case of weak coupling, which results in uniformly distributed
sites. Then we move on to show full synchronization when the sites have a well-concentrated initial distribution, and the coupling strength is sufficiently strong. In section \ref{5} we give some concluding remarks.

\section{Setting and statement of results} \label{2}

Let us consider a distribution on the circle $\mathbb{T}=\mathbb{R} \backslash \mathbb{Z}$ defined by the probability measure $\mu$. This distribution represents the initial state
of our coupled map system. Our investigation focuses on two particular cases: \begin{itemize}
\item $\mu$ is an average of a finite number of Dirac masses. It is shown below that this corresponds to a finite number of interacting sites studied in \cite{fernandez2014breaking}.
    \item $\mu$ is absolutely continuous (with respect to Lebesgue measure). This case can be regarded as an infinite-site-version of the model, and will be referred to as a \textit{continuum of sites} from now on.
\end{itemize}
    The position of each site evolves according to the doubling map and the interaction with other sites. Let
\[
F_{\mu}=d \circ \Phi_{\mu},
\]
where the individual dynamics is defined by the doubling map
\[
d:  \text{ } \mathbb{T} \to \mathbb{T}, \qquad
d(x) =2x \quad \text{mod 1}
\]
and
\[
\Phi_{\mu}:  \text{ }  \mathbb{T} \to \mathbb{T}, \qquad
\Phi_{\mu}(x)=x+\varepsilon \int_0^1 g(y-x)\text{ d}\mu(y) \qquad \text{mod 1}
\]
is the so-called coupling map, where $0 \leq \varepsilon < 1$ is the strength of interaction. The function $g$ is the signed distance on the torus, namely the lift of
\[
g(u)=\begin{cases}
0 & \text{ if } u = \pm \frac{1}{2} , \\
u& \text{ if } u \in \left(-\frac{1}{2}, \frac{1}{2} \right)
\end{cases}
\]
to $\mathbb{R}$, see figure \ref{g}.
Hence
\begin{equation} \label{main0}
F_{\mu}:  \text{ } \mathbb{T} \to \mathbb{T}, \qquad
F_{\mu}(x)=2\left(x+\varepsilon\int_{0}^1g(y-x)\text{ d}\mu(y)\right) \qquad \text{ mod 1}.
\end{equation}

Define $\mu'=(F_{\mu})_{\ast}\mu$ as the distribution of the sites after one time-step. The dynamics in this next time step will be $F_{\mu'}$. As we can see, this is not a dynamical system in the traditional sense: the dynamics will be different in each time step, and it will be defined by the current distribution of the sites.
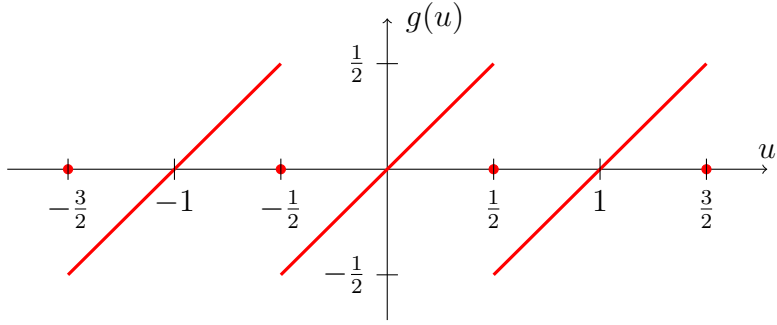
\begin{figure}[h!]
 \centering
 \begin{tikzpicture}[scale=2]
       \draw[->] (-2.5,0) -- (2.5,0) node[above] {$u$};
       \draw[->] (0,-1) -- (0,1) node[right] {\hspace{0.1cm}$g(u)$};
       \draw[very thick,red] (-0.7,-0.7) -- (0.7,0.7);
       \draw[very thick,red] (0.7,-0.7) -- (2.1,0.7);
       \draw[very thick,red] (-0.7,0.7) -- (-2.1,-0.7);
       \draw[red, fill] (0.7,0) circle (0.03cm);
       \draw[red, fill] (-0.7,0) circle (0.03cm);
       \draw[red, fill] (2.1,0) circle (0.03cm);
       \draw[red, fill] (-2.1,0) circle (0.03cm);
       \foreach \x/\xtext in {0.7/\frac{1}{2},1.4/1, 2.1/\frac{3}{2},-0.7/-\frac{1}{2},-1.4/-1, -2.1/-\frac{3}{2}}
           \draw[shift={(\x,0)}] (0pt,2pt) -- (0pt,-2pt) node[below] {$\xtext$};
       \foreach \y/\ytext in {0.7/\frac{1}{2},-0.7/-\frac{1}{2}}
             \draw[shift={(0,\y)}] (2pt,0pt) -- (-2pt,0pt) node[left] {$\ytext$};
     \end{tikzpicture}
     \caption{The function $g$} \label{g}
     \end{figure}

We first study how the distribution of the sites evolve when subjected to such dynamics. Naturally, we first ask if the obtained sequence of distributions converges in some sense to a limiting distribution. More precisely, we calculate the pushforward of the initial measure $\mu_0$ by the dynamics to get $\mu_1=(F_{\mu_0})_*\mu_0$, then calculate $\mu_2=(F_{\mu_1})_*\mu_1$ and so on. However, the limit $\lim\limits_{t \to \infty}(F_{\mu_t})_*\mu_t$ (considered in the weak topology) is not likely to exist, typically. Instead we can study the convergence of the time averages
\begin{equation} \label{time}
A(T)=\frac{1}{T}\sum_{t=0}^{T-1}\mu_t.
\end{equation}

The above questions concern the behaviour of the totality of the sites. We are further interested in capturing the phenomenon of synchronization in our coupled map system, that is, the behaviour of the sites with respect to each other. To this end we study special classes of initial distributions, and give some characterization for $\mu_T$, $T$ large. We study the full range $\varepsilon \in [0,1)$ of the coupling parameter, and observe similar behaviour to that of the uncoupled system when $\varepsilon$ is close to zero, and synchronization, when $\varepsilon$ is sufficiently large.

In the first part of this paper we study special singular initial measures, the average of $N$ point masses:
\[
\mu_0=\frac{1}{N}\sum_{i=1}^{N}\delta_{x_i}.
\]
 This is the case of finite system size, each point mass representing the position of a site. By labeling the sites, the coupled map system can be represented as a dynamical system on the $N$-dimensional torus with dynamics
 \begin{equation} \label{mainsing}
 (F_{\varepsilon,N}(x))_s=2\left(x_s+\frac{\varepsilon}{N}\sum_{r=1}^Ng(x_r-x_s) \right), \quad \forall s \in \{1,\dots,N\}, \quad  x=(x_s)_{s=1}^N \in \mathbb{T}^N.
 \end{equation}
 This is a piecewise affine map of the N-dimensional torus: the singularities (arising from the singularities of the map $g$) are the hyperplanes
 \[
 x_r-x_s = \frac{1}{2} \quad \mod 1 \qquad r,s=1,\dots,N, \thinspace (r \neq s),
 \]
 and the linear part of the map (on each domain of continuity) has eigenvalues 2 and $2(1-\varepsilon)$ with multiplicities 1 and $N-1$ respectively. Hence we say that the map is expanding if $0 \leq \varepsilon < \frac{1}{2}$ and contracting if $\frac{1}{2} < \varepsilon < 1$.

 This is in fact the system Fernandez studied in \cite{fernandez2014breaking}. We first show that for all $\varepsilon$ values for which the dynamics is expanding, the time averages of the point masses supported on the trajectory of all typical initial positions tend to uniform distribution.
 \begin{theo} \label{t1}
 Let $0 \leq \varepsilon < \frac{1}{2}$, and let us denote the Lebesgue measure on $\mathbb{T}^N$ as $\lambda^N$. The time averages $A(T)$ defined by \eqref{time} converge to $\lambda=\lambda^1$ for $\lambda^N$- a.e. starting state $x=(x_1,\dots,x_N)$.
 \end{theo}

 \begin{con}
 The statement of Theorem \ref{t1} also holds for all $\frac{1}{2} < \varepsilon \leq 1$.
 \end{con}

 We study in detail the cases of $N=2$ and $N=3$. We aim to characterize the limit behaviour of these systems by giving a geometric description of their attractors. The suitable notion of attractor (following Fernandez) is the Milnor attractor: the smallest closed set which attracts the orbit of every initial condition up to a zero Lebesgue measure set. For further reference, see \cite{buescu2012exotic} and \cite{milnor1985concept}.

 For the case of two sites, we prove the following result.

 \begin{theo} \label{t2} Let us consider the dynamical system $(F_{\varepsilon,2},\mathbb{T}^2)$.
 \begin{enumerate}[(A)]
 \item Let $0 \leq \varepsilon < 1-\frac{\sqrt{2}}{2}$. The system has a unique mixing absolutely continuous invariant measure.
 \item Let $1-\frac{\sqrt{2}}{2} \leq \varepsilon < \frac{1}{2}$. The system has a unique ergodic absolutely continuous invariant measure, which is not mixing.
 \item Let $\frac{1}{2} < \varepsilon < 1$. The Milnor attractor of the system is a circle on $\mathbb{T}^2$, on which the dynamics acts as the doubling map.
 \end{enumerate}
 \end{theo}

\begin{disc} \label{d2}
\textbf{Case (B) of Theorem \ref{t2}.} This concerns non-mixing behaviour in the expanding regime, for which we can provide a rather explicit description. For each
$\varepsilon \geq 1-\frac{\sqrt{2}}{2}$ there exists a $K=K(\eps)\geq 2$ such that $F_{\varepsilon,2}$ has a factor which is a cycle of period $K$. That is, there exist sets of positive measure $A_1,A_2,...A_{K-1}$ such that $FA_i=A_{i+1}; (i=1,...,K-2)$, $FA_{K-1}=A_1$ and $F^K$ restricted to each of the $A_i$ is mixing (here $F=F_{\varepsilon,2}$ for brevity).  The dependence of $K$ on $\eps$ is rather explicit, for details we refer to section~\ref{3.1}. This corresponds to the following behaviour: the two sites occupy an almost \textit{opposite position} on the circle, that is, they are roughly a semicircle apart. In course of a $K$ iterations, the relative position of the sites keeps changing, and then they jump back to the original -- almost opposite -- position. Meanwhile the sum of the coordinates evolves according to the doubling map -- this indicates that the configuration, as a whole, behaves chaotically. The larger $\eps$ is, the larger $K$ becomes, and the closer the sites remain to an exact opposite configuration throughout. In the limit as $\eps\to\frac12$ from below, we have $K\to\infty$, while the relative position of the sites converges to exact opposite locations.

\textbf{Case (C) of Theorem \ref{t2}.} As the last point of the theorem, we actually show that when $\frac{1}{2} < \varepsilon < 1$, the diagonal of $\mathbb{T}^2$ (viewed as the unit square with sides identified) attracts every trajectory, hence the two sites asymptotically synchronize in such a way that their \textit{positions coincide}. This is in contrast to the behaviour in case (B). On the other hand the point in which the sites are synchronized moves chaotically on $\mathbb{T}$.

\end{disc}

 In the case of $N=3$, we prove the following statement. Throughout, by ``ergodic component'' we mean (the support of) an ergodic absolutely continuous invariant measure.

 \begin{theo} \label{t3} Let us consider the dynamical system $(F_{\varepsilon,3},\mathbb{T}^3)$.
 \begin{enumerate}[(A)]
  \item Let $0 \leq \varepsilon < 1-\frac{\sqrt{2}}{2}$. The system has a unique ergodic absolutely continuous invariant measure.
  \item If $\frac{4-\sqrt{10}}{2} \leq \varepsilon < \frac{1}{2}$, there are at least six ergodic components.
  \item If $\frac{1}{2} < \varepsilon < 1$, the Milnor attractor is the union of three circles on $\mathbb{T}^3$, one is invariant and the other two map onto each other.
  \end{enumerate}
 \end{theo}

\begin{con}
We conjecture that there is a unique ergodic absolutely continuous invariant measure for $\eps< \frac{4-\sqrt{10}}{2}$, and that the number of ergodic components is precisely six for $\frac{4-\sqrt{10}}{2}\le \eps<\frac12$.
\end{con}

\begin{rem}
The values $1-\frac{\sqrt{2}}{2}$ and $\frac{4-\sqrt{10}}{2}$ are not new, these were already identified by Fernandez in \cite{fernandez2014breaking}.
In fact, the only new result that we prove is the emergence of six ergodic components at $\frac{4-\sqrt{10}}{2}$, while \cite{fernandez2014breaking}
showed the presence of (at least) two components, and detected the splitting into six components at a strictly higher value. Nonetheless, we obtain our results
with a different representation of the system, which allows for a more direct interpretation of the dynamical phenomena in terms of the behaviour of the sites, see Discussion~\ref{d3} below.
\end{rem}

\begin{figure}
\centering
\begin{tikzpicture}
\draw (-3,2) circle (2cm);
\filldraw (-3+1,2+1.73) circle (0.08cm);
\filldraw (-3+2,2) circle (0.08cm);
\filldraw (-3-1,2-1.73) circle (0.08cm);
\draw (5,2) circle (2cm);
\filldraw (5-2,2) circle (0.08cm);
\filldraw (5-1,2+1.73) circle (0.08cm);
\filldraw (5+1,2-1.73) circle (0.08cm);
\draw (-3,-1) node {Case (1)};
\draw (5,-1) node {Case (2)};
\draw (-3+1.25,2+2.16) node {A};
\draw (-3+2.5,2) node {B};
\draw (-3-1.25,2-2.16) node {C};
\draw (5-2.5,2) node {B};
\draw (5-1.25,2+2.16) node {A};
\draw (5+1.25,2-2.16) node {C};
\end{tikzpicture}
\caption{Case (1) depicts a configuration belonging to an odd component ($I$, $III$ or $V$), Case (2) depicts a configuration belonging to an even component ($II$, $IV$ or $VI$)} \label{korok}
\end{figure}
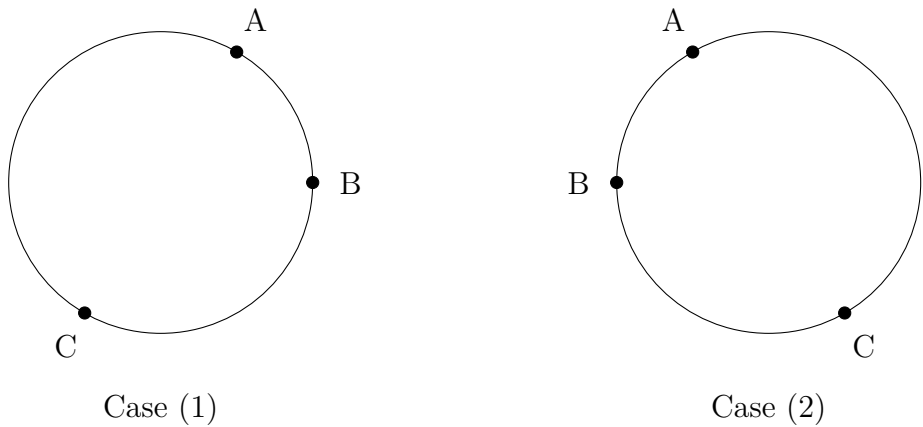

\begin{disc}\label{d3}
Let us point out first that in all cases below the relative position of the three sites is described. The sum of the three coordinates always evolves according to the doubling map on the circle -- indicating that the configuration as a whole behaves chaotically.

\textbf{Case (B) of Theorem \ref{t3}.} Let us denote the six ergodic components by $I, II, III, IV, V$ and $VI$. In fact, these ergodic components consist of
two connected components each, for instance $I=Ia\cup Ib$. Thus altogether $12$ components can be distinguished, all of which correspond to a certain type of
distribution of the sites, which we describe here. Let us consider the locations of the sites as three points on $\mathbb{T}$, which split the circle into three arcs. Irrespective of which site they correspond to, we label the three points by the symbols $A$, $B$ and $C$ according to the following convention: the arc between points
$A$ and $B$ is the shortest, and the one between $A$ and $C$ is the longest of the three arcs (note that for Lebesgue almost every configuration the three lengths are distinct).

Now for odd numbered components ($I, III$ and $V$) the three points $A$, $B$ and $C$ follow each other in a clockwise order, while for even numbered components ($II, IV$ and $VI$) the three points $A$, $B$ and $C$ follow each other in a counterclockwise order on $\mathbb{T}$, see figure \ref{korok}. We may refer to these two configurations as odd and even. These are in one-to-one relation by the reflection about the origin. This provides further evidence that the loss of ergodicity corresponds, in part, to the \textbf{breaking of the inversion symmetry}.

Now, further distinctions among the components correspond to the possible identifications of the three sites $x$, $y$ and $z$ with the points $A$, $B$ and $C$. For instance, if for an odd configuration $y$ is located at $A$, $x$ is located at $B$ and $z$ is located at $C$, a configuration in $Ia$ is obtained. If the sites at points $B$ and $C$ are exchanged -- that is, in this new realization $x, y$ and $z$ are located at $C, A$ and $B$, respectively --  a point in component $Ib$ is obtained. For the same odd configuration, the remaining four permutations map to the realizations in $IIIa, IIIb, Va$ and $Vb$. Thus the different odd components are in one-to-one relation by relabeling the sites (and the same holds for even components, too).  This provides further evidence that the loss of ergodicity corresponds, in part, to the \textbf{breaking of the permutation symmetry}.

Equivalently, the six ergodic components can be described in terms of the quasimetric $d$ on $\mathbb{T}$ such that $d(x,y)$ is the length of the counterclockwise arc from
$x$ to $y$; in particular, the six components correspond to the six possible orders of $d(x,y)$, $d(y,z)$ and $d(z,x)$. For details we refer to section~\ref{3}. We further observe by studying the Milnor  attractor in the expanding case that the sites cannot get closer to each other than $\frac{\varepsilon}{3}$.

\textbf{Case (C) of Theorem \ref{t3}.} In this contracting case ($\frac{1}{2} < \varepsilon < 1$), we show that the states asymptotically achieved by the system are of two types. Either $x=y=z$ on $\mathbb{T}$, that is, the site locations are \textbf{fully synchronized}; or the three points are \textbf{evenly placed} on $\mathbb{T}$, that is, all three distances among the three site locations are $1/3$ on $\mathbb{T}$ . These may be thought of as attractive and repulsive asymptotic states, the basins of attraction for which can be identified. In particular, if the initial distribution of the three sites is sufficiently concentrated on $\mathbb{T}$, then asymptotically there is attractive synchronization, and if the initial distribution is sufficiently even on $\mathbb{T}$, then asymptotically there is repulsive synchronization.
\end{disc}

% Let $d$ be the quasimetric on $\mathbb{T}$ such that $d(x,y)$ is the length of the counterclockwise arc from $x$ to $y$.
% \begin{rem}
% If the sites have states $x,y$ and $z$ on $\mathbb{T}$,  the six ergodic components correspond to the six possible orders of $d(x,y)$, $d(y,z)$ and $d(z,x)$. These ergodic %components can be grouped into two groups of three, such that each component in a group can be mapped onto every other in the same group by relabeling the sites.
% Hence when we look at the original unlabeled system, only two invariant sets of states exist, and every state in the first group has a corresponding state in the second group, %such that the two corresponding states are the mirror image of each other with respect to the origin.
% \end{rem}

 The case of larger system size is difficult to analyze due to geometric complexity, but simulations lead us to believe that breaking of ergodicity also occurs in the expanding regime for all $N>3$.

 \begin{con}
 For all $N \in \mathbb{N}$ the system $(F_{\varepsilon,N},\mathbb{T}^N)$ is ergodic if $\varepsilon < \varepsilon_0$, for some $0 < \varepsilon_0=\varepsilon_0(N)$ and has multiple ergodic components when $\varepsilon_1 \leq \varepsilon$ for some $\varepsilon_1=\varepsilon_1(N) < \frac{1}{2}$.
 \end{con}

In the second part of this paper we will deal with \textit{absolutely continuous initial measures}, this corresponds to studying \textit{coupled map systems with a continuum of sites}. If d$\mu=f$d$\lambda$, the dynamics are
\[
F_{f}(x)=2x+2\varepsilon \int_0^1g(y-x)f(y)\text{ d}y \qquad \text{ mod 1}.
\]
Even though the dynamics changes from step to step, it keeps some main characteristic features. We can easily see that our one-step dynamics $F_{f}$ is continuous and homotopic to the doubling map for any probability density $f$. According to Franks \cite{franks1968anosov}, if $\ell$ is a continuous circle map homotopic to the doubling map $d$, then there exists a continuous, onto map $\alpha: \mathbb{T} \to \mathbb{T}$ with
\[
\alpha \circ \ell = d \circ \alpha.
\]
Hence each such map $F_{f}$  is semiconjugate to the doubling map, thus bears a strong similarity to it (for example, $F_{f}$ is always a degree two covering map, always has periodic points of each period etc.)

Now analysis narrows down to the study of the evolution of the densities. We will assume some smoothness for the initial density, typically $C$ or $C^1$ (in the topology of $\mathbb{T}$). For the transfer operator of the dynamics $F_f$ we are going to use the notation $\mathcal{L}_f$.

We first study the case when the \textbf{coupling is weak}, and the sites are initially nearly uniformly distributed. In this case we can show that their \textbf{distribution will tend to uniform} at an exponential rate. This is the same behaviour that the uncoupled system exhibits, since in the case of $\varepsilon=0$, $F_f(x)=2x \mod 1$ for all time steps.

Let us denote the total variation on $\mathbb{T}$ by $|\cdot|_{TV}$.

\begin{theo} \label{t4}
Let $f_0\in C^1(\mathbb{T})$ be the density of the initial measure $\mu_0$ with $|f_0|_{TV} \leq \delta$. Assume that $\varepsilon > 0$ is such that
\[
\varepsilon < \frac{1}{1+4\delta}.
\]
Then the density $f_1$ of the pushforward measure $(F_{f_0})_{\ast}\mu_0$ is in $C^1(\mathbb{T})$ and has total variation
\[
|f_1|_{TV} \leq c |f_0|_{TV}
\]
for some $c < 1$.

This implies that $f_n\to 1$ exponentially in the total variation metric. Hence, asymptotically, the distribution of the sites is uniform on $\mathbb{T}$.
\end{theo}

We then turn to the investigation of \textit{strong coupling and synchronization}. We show two results concerning initial distributions which concentrate an essential part of the mass to a small subset of $\mathbb{T}$. We show that in this case the distribution converges to a point mass moving chaotically on the torus if $\varepsilon$ is sufficiently large.

 Let $d$ be the quasimetric on $\mathbb{T}$ such that $d(x,y)$ is the length of the counterclockwise arc from $x$ to $y$. Suppose $f$ is supported on an interval of the torus $[b_1,b_2]$ or $[0,b_2] \cup [b_1,1]$ with $d(b_1,b_2) \leq \frac{1}{2}$.

 If $\supp f \subseteq [b_1,b_2]$ we define the center of mass as
\[
M(f)=\int_{b_1}^{b_2}yf(y)\text{ d}y.
\]

If $\supp f \subseteq [0,b_2] \cup [b_1,1]$, we define $M(f)$ as
\[
M(f)=\int_0^{b_2}(y+1)f(y)\text{ d}y+\int_{b_1}^{1}yf(y)\text{ d}y \qquad \text{ mod } 1,
\]
see figure \ref{cm}.
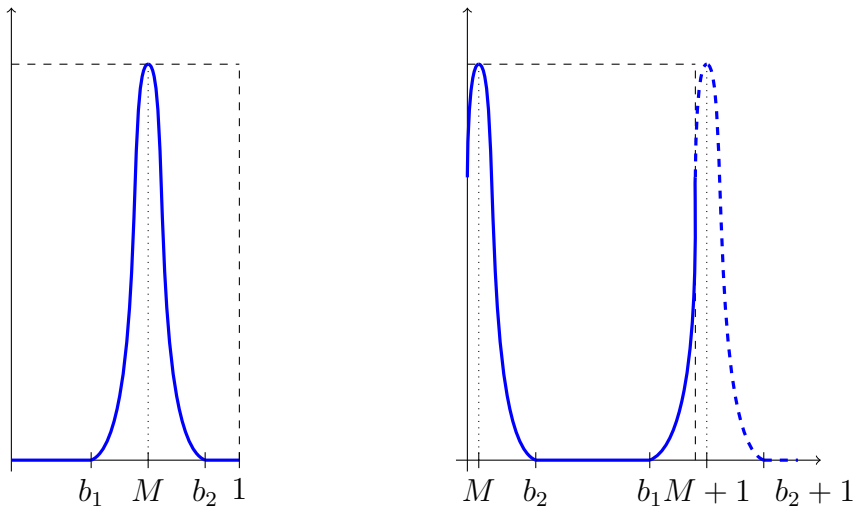
\begin{figure}
\centering
\begin{tikzpicture}[scale=1.5]
       \draw[shift={(-2,0)}] (0,0) -- (2,0);
       \draw[->,shift={(-2,0)}] (0,-0.1) -- (0,4);
       \draw[dashed,shift={(-2,0)}] (2,0) -- (2,3.5);
       \draw[dashed,shift={(-2,0)}] (0,3.5) -- (2,3.5);
       \draw[dotted,shift={(-2,0)}] (1.2,0) -- (1.2,3.5);
       \draw[very thick, blue,shift={(-2,0)}] (0,0) -- (0.7,0);
       \draw[very thick, blue,shift={(-2,0)}] (1.7,0) -- (2,0);
       \draw[very thick, blue,shift={(-2,0)}] (0.7,0) .. controls (1.2,0.2) and (1,3.5) .. (1.2,3.5);
       \draw[very thick, blue,shift={(-2,0)}] (1.7,0) .. controls (1.2,0.2) and (1.4,3.5) .. (1.2,3.5);
       \foreach \x/\xtext in {1.2/M,2/1,1.7/b_2,0.7/b_1}
           \draw[shift={(\x,0)},shift={(-2,0)}] (0pt,2pt) -- (0pt,-2pt) node[below] {$\xtext$};

       \draw[->,shift={(2,0)}] (-0.1,0) -- (3.1,0);
              \draw[->,shift={(2,0)}] (0,-0.1) -- (0,4);
              \draw[dashed,shift={(2,0)}] (2,0) -- (2,3.5);
              \draw[dashed,shift={(2,0)}] (0,3.5) -- (2,3.5);
              \draw[dotted,shift={(0.9,0)}] (1.2,0) -- (1.2,3.5);
              \draw[dotted,shift={(2.9,0)}] (1.2,0) -- (1.2,3.5);
              \draw[very thick, blue,shift={(2.9,0)}] (-0.3,0) -- (0.7,0);
              \draw[very thick,dashed, blue,shift={(2.9,0)}] (1.7,0) -- (2,0);
              \draw[very thick, blue,shift={(2.9,0)}] (0.7,0) .. controls (1.2,0.2) and (1.08,2.5) .. (1.1,2.5);

              \draw[very thick,dashed, blue,shift={(2.9,0)}] (1.7,0) .. controls (1.2,0.2) and (1.4,3.5) .. (1.2,3.5);

              \draw[very thick, blue,shift={(0.9,0)}] (1.7,0) .. controls (1.2,0.2) and (1.4,3.5) .. (1.2,3.5);

              \draw[very thick, blue,shift={(0.9,0)}] (1.1,2.5) .. controls (1.1,2.5) and (1.1,3.5) .. (1.2,3.5);

              \draw[very thick,dashed, blue,shift={(2.9,0)}] (1.1,2.5) .. controls (1.1,2.5) and (1.1,3.5) .. (1.2,3.5);
              \foreach \x/\xtext in {2.1/M}
                            \foreach \x/\xtext in {2.1/M,4.1/M+1,2.6/b_2,3.6/b_1}
                  \draw[shift={(\x,0)}] (0pt,2pt) -- (0pt,-2pt) node[below] {$\xtext$};

              \foreach \x/\xtext in {2.1/M}
                                          \foreach \x/\xtext in {4.6/b_2+1}
                                \draw[shift={(\x,0)}] (0pt,2pt) -- (0pt,-2pt) node[below right] {$\xtext$};
     \end{tikzpicture}
     \caption{Definition of the center of mass.} \label{cm}
\end{figure}
\begin{theo} \label{t5} Let $f_0 \in C(\mathbb{T})$ with
 $\supp f_0 \subseteq [b_1,b_2]$  or
 $\supp f_0 \subseteq [0,b_2] \cup [b_1,1]$ with $d(b_1,b_2) \leq \frac{1}{2}$.

The density $f_1=\mathcal{L}_{f_0}f_0$ is in $C(\mathbb{T})$ and has support $\supp f_1 \subseteq [b_1',b_2']$,  or $\supp f_1 \subseteq [0,b_2'] \cup [b_1',1]$, such that
\begin{itemize}
\item $d(b'_1,b'_2)=2(1-\varepsilon)d(b_1,b_2)$,
\item  $\sup f_1=\frac{\sup f_0}{2(1-\varepsilon)}$,
\item $M(f_1)=2M(f_0) \mod 1$.
\end{itemize}
\end{theo}

\begin{disc}\label{d4}
Note that this theorem applies to any value of $\eps$ if the conditions on the support of the initial density $f$ are satisfied. Nonetheless, the important corollary corresponds to the strongly coupled case $\frac{1}{2} < \varepsilon$, when the support shrinks, and the process can be iterated. Hence we obtain in this case that the support of the distribution shrinks to zero at an exponential rate, its supremum tends to infinity, hence \textbf{the distribution converges to a point mass} -- the positions of all sites synchronize. But because the center of mass evolves according to the doubling map, the time averages defined by \eqref{time} will tend to Lebesgue, the system will remain chaotic in this sense.

Let us keep on discussing the  case $\frac{1}{2} < \varepsilon$. In this regime, on the one hand, for sufficiently even initial distributions Theorem~\ref{t4} applies, that is,
asymptotically the sites are uniformly distributed on $\mathbb{T}$. On the other hand, for sufficiently concentrated initial distributions, Theorem~\ref{t5} applies,
that is, the sites are fully synchronized asymptotically. This is analogous to the phenomena observed in the attracting regime for the case of finitely many sites, see Discussion~\ref{d3}.
\end{disc}

%We can further show that if the support is larger, but the coupling %is stronger, a similar result can hold.
We further show that synchronization can take place, even if the support of the initial distribution is larger, provided that the coupling is strong enough.

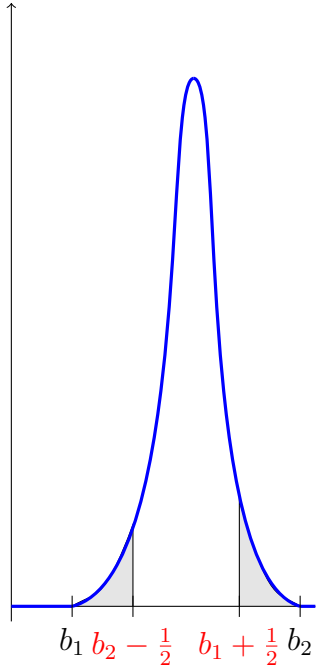
\begin{figure}[h!]
  \centering
  \begin{tikzpicture}[scale=2]
  \draw [fill=gray,fill opacity=0.2] (0.8,0) -- (0.8,0.54) .. controls (0.6,0) and (0.55,0.1) .. (0.4,0);
  \draw [fill=gray,fill opacity=0.2] (1.5,0) -- (1.5,0.7) .. controls (1.7,0) and (1.75,0.1) .. (1.9,0);
        \draw (0,0) -- (2,0);
        \draw[->] (0,-0.1) -- (0,4);
        \draw[very thick, blue] (0,0) -- (0.4,0);
        \draw[very thick, blue] (1.9,0) -- (2,0);
        \draw[very thick, blue] (0.4,0) .. controls (1.2,0.2) and (1,3.5) .. (1.2,3.5);
        \draw[very thick, blue] (1.9,0) .. controls (1.2,0.2) and (1.4,3.5) .. (1.2,3.5);
 %       \draw[very thick, blue] (0.35,0) .. controls (0,0.4) and (0.2,1.8) .. (0,3.5);
        \foreach \x/\xtext in {0.4/b_1,1.9/b_2, 0.8/\textcolor{red}{b_2-\frac{1}{2}},1.5/\textcolor{red}{b_1+\frac{1}{2}}}
            \draw[shift={(\x,0)}] (0pt,2pt) -- (0pt,-2pt) node[below] {$\xtext$};
      \end{tikzpicture}
      \caption{The shaded area is assumed to be $< \frac{1}{4}$.} \label{shade}
      \end{figure}

\begin{theo} \label{t6}
Let $0 \leq b_1 < \frac{1}{2} <  b_2 \leq 1$ such that $\frac{1}{2} < |b_2-b_1| < 1$ and let $f_0 \in C(\mathbb{T})$, such that $\supp f_0 \subseteq [b_1,b_2]$. Let us assume that
\begin{equation}
\tag{$T1$}
\mathcal{C}=\int_0^{b_2-\frac{1}{2}}f_0(y)\text{ d}y+\int_{b_1+\frac{1}{2}}^1f_0(y)\text{ d}y < \frac{1}{4}.
\end{equation}
Let
\[
\frac{(b_2-b_1)-\frac{1}{4}}{(b_2-b_1)-\mathcal{C}} \leq \varepsilon < 1.
\]
 The density $f_1=\mathcal{L}_{f_0}f_0 \in C(\mathbb{T})$ has support contained in some interval  of $\mathbb{T}$  $[0,b_2'] \cup [b_1',1]$ such that  $d(b'_1,b'_2) \leq \frac{1}{2}$.
\end{theo}

Assuming that an essential part of the mass is supported on a small subset of $\mathbb{T}$ (for an illustration see figure \ref{shade}), this theorem states that in one step the support shrinks to an interval on the torus with length less than $\frac{1}{2}$ when the coupling is sufficiently strong, and then our previous analysis will apply.

So both in the case of finite system size and a system with a continuum of sites, we observed that for some classes of initial distributions corresponding to suficciently concentrated initial states, synchronization occurs when the coupling is strong.

Observe that in all discussed classes of absolutely continuous initial measures, the time averages \eqref{time} converge once again to $\lambda$ for typical initial measures. We conjecture that this is true in a much greater generality.

\section{Finite system size} \label{3}

In this section we consider an initial measure concentrated on finitely many points of the torus. Let $x_1,\dots,x_n \in \mathbb{T}$ and
\[
\mu_0=\frac{1}{N}\sum_{i=1}^{N}\delta_{x_i}.
\]
This measure can be thought of as describing a system of $N$ identical sites with initial state $x_1,\dots,x_n \in \mathbb{T}$.
In this case the dynamics is
\begin{equation} \label{diszkret}
F_{\mu_0}(x)=2\left(x+\frac{\varepsilon}{N}\sum_{i=1}^{n}g(x_i-x)\right) \qquad \text{ mod 1},
\end{equation}
and the pushforward measure is
\[
\mu_1=(F_{\mu_0})_{\ast}\mu_0=\frac{1}{N}\sum_{i=1}^{N}\delta_{F_{\mu_0}(x_i)}.
\]
To see this, let $A \subseteq \mathbb{T}$ be measurable. Then
\[
(F_{\mu_0})_{\ast}\mu_0(A)=\mu_0(F_{\mu_0}^{-1}(A))=\frac{1}{N}\sum_{i=1}^{N}\delta_{x_i}(F_{\mu_0}^{-1}(A))=\frac{1}{N}\sum_{i=1}^{N}\delta_{F_{\mu_0}(x_i)}(A).
\]
This means that the pushforward measure is supported on the image of the initial points by the map $F_{\mu_0}$.

When $\varepsilon=0$, each site evolves independently according to the doubling map. In this case the distribution of the sites in step $t$ is
\[
\mu_t=\frac{1}{N}\sum_{i=1}^N\delta_{d^t(x_i)}.
\]
Since the doubling map is well known to be ergodic with respect to the Lebesgue measure on $\mathbb{T}$, we see that for Lebesgue--almost every initial position of the $N$ sites the time averages $A(T)$ defined by \eqref{time} converge to the Lebesgue measure on the torus.

Now when $\varepsilon \neq 0$, the dynamics is more complicated than the doubling map. We are going to label the sites, and study the following $N$ dimensional system instead. Recalling equation \eqref{mainsing}, the dynamics is given by
\[
(F_{\varepsilon,N}(x))_s=2\left(x_s+\frac{\varepsilon}{N}\sum_{r=1}^Ng(x_r-x_s) \right), \quad \forall s \in \{1,\dots,N\}, \quad  x=(x_s)_{s=1}^N \in \mathbb{T}^N,
\]
denoting the $N$-dimensional torus $\mathbb{R}^N\backslash \mathbb{Z}^N$ by $\mathbb{T}^N$.

We are going to state an observation, which will play an important role in the proofs of Theorems \ref{t1} -- \ref{t3}.

\begin{obs} \label{obs1}
Let us think of $\mathbb{T}^N$ as the $N$ dimensional unit hypercube with opposite sides identified. Notice that the circles parallel to the main diagonal of the $N$ dimensional hypercube, $$(x_1,\dots,x_N)+(t,\dots,t) \mod 1, \qquad (x_1,\dots,x_n) \in \mathbb{T}^N, \quad  t \in \mathbb{R}$$ are mapped onto each other without being cut by singularities (since any such circle is either disjoint of a singularity or contained completely by it). Moreover, the effect of the map $F_{\varepsilon,N}$ on these circles is the doubling map: they will be stretched to twice their size, and mapped onto another such circle covering it exactly twice. From this, we can see that ergodic components must contain full circles, and in fact the basins of the ergodic invariant measures also contain full circles.
\end{obs}

\begin{proof}[Proof of Theorem \ref{t1}.] We are going to show that for arbitrary finite system size, the time averages defined by \eqref{time} converge to Lebesgue not just for $\varepsilon=0$, but for all $0 \leq \varepsilon < \frac{1}{2}$.

$F_{\varepsilon,N}$ is a piecewise affine map, whose Jacobian has eigenvalues 2 with multiplicity 1, and $2(1-\varepsilon)$ with multiplicity $N-1$. Hence, by our assumption  $0 \leq \varepsilon < \frac{1}{2}$, the map is fully expanding. According to Thomine \cite{thomine2010spectral}, such maps have a finite number of ergodic absolutely continuous invariant measures (acim-s), whose basins cover Lebesgue almost all of $\mathbb{T}^{N}$ (see also \cite{saussol2000absolutely} for details on existence).

We claim that all invariant densities are constant on the 'diagonal' circles mentioned in Observation~\ref{obs1}. Let us consider one of the ergodic acim-s and its density, to be denoted by $\mu_f$ and $f$ in the sequel.
Define another density $f_0$, which is supported on the circles contained in the basin of $\mu_f$, and constant on each supporting circle.

Notice that the value of the pushforward density at each point $x \in \mathbb{T}^N$ only depends on which circles $x$ has preimages, and this only depends on which circle $x$ is on. In conclusion, the density obtained by pushing $f_0$ forward  will be also constant on such circles. Iterating this argument we see that the time averages obtained from these pushforward densities will also be constant on such circles. By using the ergodicity of $\mu_f$, the time averages of $f_0$ converge to $f$,
and we see that $f$ is also constant on such circles as claimed above.

It follows that the density of the $N$-dimensional invariant measure $\mu_f$ will have $N-1$ dimensional marginals which are constant on circles parallel to the main diagonal of the $N-1$ dimensional hypercube. By induction on the dimension we see that the one-dimensional marginals of $\mu_f$ are Lebesgue. We are going to denote these marginals by $(\mu_f)_i$.

Now let
 \[
 \mathcal{A}(T)=\frac{1}{T}(\delta_{x_1,\dots, x_N}+\dots+\delta_{F_{\varepsilon,N}^{T-1}(x_1,\dots, x_N)}),
 \]
 where $\delta_{x_1,\dots, x_N}$ is the Dirac-measure supported on $(x_1,\dots, x_N) \in \mathbb{T}^N$. Because of the ergodicity of $\mu_f$, this measure converges to $\mu_f$ in the weak topology for $\mu_f$-almost every $x=(x_1,\dots,x_N) \in \mathbb{T}^N$. By restricting the convergence to functions only depending on the variable $x_i$, we see that the marginals $\mathcal{A}(T)_i$ converge to $(\mu_f)_i=\lambda$. Observe that
 \[
 \mathcal{A}(T)_i=\frac{1}{T}(\delta_{x_i}+\dots+\delta_{F^{T-1}(x_i)}),
 \]
 where $F^{T-1}=F_{\mu_{T-2}}\dots F_{\mu_1}F_{\mu_0}$, and $F_{\mu_t}$ is defined by \eqref{diszkret}. Hence we can conclude that the time averages
 \[
 A(T)=\frac{1}{T}\sum_{t=0}^{T-1} \mu_t=\frac{1}{T}\sum_{t=0}^{T-1}\left(\frac{1}{N}\sum_{i=1}^{N}\delta_{x_i}+\frac{1}{N}\sum_{i=1}^{N}\delta_{F(x_i)}+\dots+ \frac{1}{N}\sum_{i=1}^{N}\delta_{F^{T-1}(x_i)} \right)=\frac{1}{N}\sum_{i=1}^N \mathcal{A}(T)_i
 \]
 also converge to the Lebesgue measure for $\mu_f$-almost every $x=(x_1,\dots,$ $x_N)$ $\in \mathbb{T}^N$. Now this is true for each of the finitely many ergodic invariant measures, and we mentioned that their basins cover Lebesgue almost all of $\mathbb{T}^N$. By definition, the basin of a measure $\mu$ is the set of points $x \in \mathbb{T}^N$ for which $\frac{1}{T}\sum_{t=0}^{T-1}\delta_{F^t x}$ converges to the measure $\mu$. Therefore, the convergence of $A(T)$ to Lebesgue actually holds for all initial conditions in the basins of the invariant measures, hence for $\lambda^N$-a.e.~$x \in \mathbb{T}^N$.
\end{proof}

Observe that since $g$ is an odd function,
 \[
 \sum_{s=1}^N (F_{\varepsilon,N}(x))_s=2\sum_{s=1}^N x_s \qquad \text{mod 1}.
 \]
 So we see the quantity $m=\sum_{s=1}^N x_s$ evolves according to the doubling map. Let us define new coordinates on the torus as
 \begin{align*}
 u_1&=\sum_{s=1}^N x_s \\
 u_{i+1}&=x_i-x_{i+1}, \quad i=1,\dots, N-1
 \end{align*}
 Let us denote $u=(u_1,\dots,u_N)$ for brevity. The system
 \begin{align*}
 (G_{\varepsilon,N}(u))_1&=\sum_{s=1}^N 2x_s \quad \text{mod 1} \\
 (G_{\varepsilon,N}(u))_{i+1}&=(F_{\varepsilon,N}(x))_i-(F_{\varepsilon,N}(x))_{i+1} \quad \text{mod 1}, \quad i=1,\dots, N-1
 \end{align*}
 is a factor of the original system defined by $F_{\varepsilon,N}$, since $\{x,x+\frac{1}{N},\dots,x+\frac{N-1}{N}\}$ share the same $u$-coordinates. This map $G_{\varepsilon,N}$ has a direct product structure: $G_{\varepsilon,N}$ acts as the doubling map in the direction of the first coordinate, and independently in the perpendicular direction.

 Since the mapping $G_{\varepsilon,N}|_{u_2,\dots,u_N}$ is in fact a function of the differences of the positions of the sites, it will prove quite useful in studying synchronization phenomena. The Jacobian of this mapping is $2(1-\varepsilon)I$, where $I$ is the identity matrix of size $N-1$. So this map is piecewise expanding, if $0 \leq \varepsilon < \frac{1}{2}$, and contracting if $\frac{1}{2} < \varepsilon < 1$. In the next two subsections, we are going to study in detail a system with 2 and 3 sites, that is when $G_{\varepsilon,N}|_{u_2,\dots,u_N}$ is a mapping of $\mathbb{T}$ and $\mathbb{T}^2$, respectively.

\subsection{\boldmath$N=2$} \label{3.1}
In this subsection, we are going to give the proof of Theorem \ref{t2} and provide background for Discussion~\ref{d2}.
%stating that the system has a mixing absolutely continuous invariant measure if the coupling is sufficiently weak, an ergodic, but not mixing absolutely continuous invariant %measure if the coupling is stronger, but the dynamics is still expanding, and an invariant circle attracts all trajectories, if the coupling is strong. We conclude by interpreting %the results for the coupled map system of size 2.

\begin{proof}[Proof of Theorem \ref{t2}.]
We first give a geometric description of the Milnor attractor, then prove that it supports a unique ergodic acim in the expanding case $(0 \leq \varepsilon < \frac{1}{2}$). We prove that this measure is mixing if and only if $0 \leq \varepsilon <  1-\frac{\sqrt{2}}{2}$. We finish this section by commenting on the contracting case $\frac{1}{2} < \varepsilon < 1$.

The map $F_{\varepsilon, N}$ takes the form
 \begin{equation}
 F_{\varepsilon,2}(x,y)=(2x+\varepsilon g(y-x), 2y+\varepsilon g(x-y)) \quad \text{ mod } 1, \quad x,y \in \mathbb{T}.
 \end{equation}

% The existence of the unique mixing acim, when $\varepsilon$ is sufficiently small is the consequence of some general arguments. Let us have a coupled map system of finite size $N$ with a piecewise $C^2$ site dynamics $\tau$ such that $\inf \tau' > 2$ and a (1,0)-coupling -- this means that the coupling map is of the form
%   \[
%   \Phi_{\varepsilon}(\mathbf{x})=\mathbf{x}+A_{\varepsilon}(\mathbf{x})
%   \]
%   such that
%   \[
%   |(A_{\varepsilon})_p| \leq 2\varepsilon, \quad
%   |(DA_{\varepsilon})_p| \leq 2\varepsilon, \quad
%   |\partial_k(DA_{\varepsilon})_{pq}| =0,
%   \]
%   for all $p,q,k=1,\dots,N$. Keller and Liverani \cite{keller2006uniqueness} state that in this case there exists an $\varepsilon_0 > 0$ such that for each $0 \leq \varepsilon < \varepsilon_0$, the coupled map system has a unique invariant measure $\mu_{\varepsilon}$ which is exponentially mixing and it is the SRB measure of the system.
%
%   In our case, it is easy to see that we have a (1,0)-coupling, since
%   \[
%   (A_{\varepsilon}(x_1,x_2))_p=\varepsilon\sum_{r=1}^2g(x_r-x_p),
%   \]
%   hence
%   \[
%    |(A_{\varepsilon})_p| \leq \varepsilon, \quad
%    |(DA_{\varepsilon})_p| \leq 2\varepsilon \quad
%    |\partial_k(DA_{\varepsilon})_{pq}| =0,
%    \]
%   for all $p,q,k=1,2$. The site dynamics is the doubling map, hence the expansion factor is uniformly 2, the slight modifications needed in the proof of the above theorem were treated in \cite{kunzle1993invariante}.

 The factor map $G_{\varepsilon,2}$ takes the form
 \begin{equation}
 G_{\varepsilon, 2}(u,v)=(2u, 2v-2\varepsilon g(v) ) \quad \text{ mod }1, \quad (u,v) \in \mathbb{T}^2.
 \end{equation}

Since $u=x+y \mod 1$ and $v=x-y \mod 1$, the points with coordinates $(x,y)$ and $\left(x+\frac{1}{2},y+\frac{1}{2} \right) \mod 1$ share the same $(u,v)$-coordinates. %Also notice that $F_{\varepsilon, 2}(x,y)=$ $F_{\varepsilon, 2}\big(x+\frac{1}{2},$ $y+\frac{1}{2} \big)$.

  Let us first study the map $G_{\varepsilon,2}$. The map of the first coordinate is the doubling map for which Lebesgue is an ergodic invariant measure. The map $v \mapsto 2v-2\varepsilon g(v)$ mod $1$ is more complicated. For the sake of simplicity, let us denote this mapping of $\mathbb{T}$ by $H$, in detail it takes the form
 \begin{equation} \label{h}
  H(v)=
  \begin{cases}
     2(1-\varepsilon)v & \text{if }  0 \leq v < \frac{1}{2} \\
     1 & \text{if } v=\frac{1}{2} \\
     2(1-\varepsilon)v+2\varepsilon-1       & \text{if } \frac{1}{2} < v \leq 1
    \end{cases}
  \end{equation}
 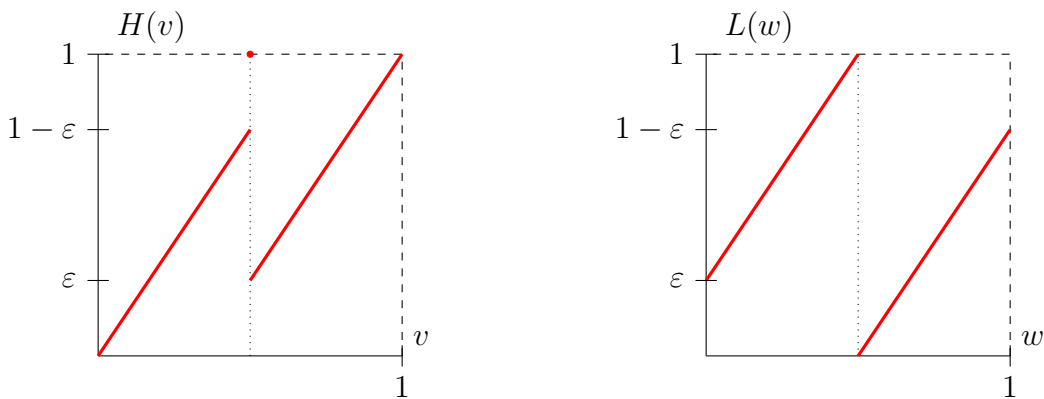
\begin{figure}
 \centering
 \begin{tikzpicture}[scale=2]
 \draw (0,0) -- (2,0) node[above right] {$w$};
        \draw (0,0) -- (0,2) node[above right] {\hspace{0.1cm}$L(w)$};
        \draw[dashed] (2,0) -- (2,2) -- (0,2);
        \draw[dotted] (1,0) -- (1,2);
        \draw[very thick,red] (0,0.5) -- (1,2);
        \draw[very thick,red] (1,0) -- (2,1.5);
        \foreach \x/\xtext in {2/1}
            \draw[shift={(\x,0)}] (0pt,2pt) -- (0pt,-2pt) node[below] {$\xtext$};
        \foreach \y/\ytext in {2/1,1.5/1-\varepsilon,0.5/\varepsilon}
              \draw[shift={(0,\y)}] (2pt,0pt) -- (-2pt,0pt) node[left] {$\ytext$};

 \draw[shift={(-4,0)}] (0,0) -- (2,0) node[above right] {$v$};
         \draw[shift={(-4,0)}] (0,0) -- (0,2) node[above right] {\hspace{0.1cm}$H(v)$};
         \draw[dashed,shift={(-4,0)}] (2,0) -- (2,2) -- (0,2);
         \draw[dotted,shift={(-4,0)}] (1,0) -- (1,2);
         \draw[very thick,red,shift={(-4,0)}] (0,0) -- (1,1.5);
         \filldraw[red,shift={(-4,0)}] (1,2) circle (0.02cm);
         \draw[very thick,red,shift={(-4,0)}] (1,0.5) -- (2,2);
         \foreach \x/\xtext in {2/1}
             \draw[shift={(\x,0)},shift={(-4,0)}] (0pt,2pt) -- (0pt,-2pt) node[below] {$\xtext$};
         \foreach \y/\ytext in {2/1,1.5/1-\varepsilon,0.5/\varepsilon}
               \draw[shift={(0,\y)},shift={(-4,0)}] (2pt,0pt) -- (-2pt,0pt) node[left] {$\ytext$};
 \end{tikzpicture}
 \caption{The map $H$ and $L$ for $0 < \varepsilon < \frac{1}{2}$.} \label{hl}
 \end{figure}

 We can redefine this map such that $H\left(\frac{1}{2}\right)=\varepsilon$, since the trajectory of a point is a zero Lebesgue measure set and is irrelevant in our analysis.
 Notice that no exterior point enters the intervals $[0,\varepsilon]$ and $[1-\varepsilon,1]$, and the trajectory of every point leaves the intervals $[0,\varepsilon]$ and $[1-\varepsilon,1]$ eventually due to expansion. Let us restrict our map $H$ to the interval $[\varepsilon, 1-\varepsilon]$, since the Milnor attractor must be a subset of this interval. By rescaling this restricted map, we get a centrally symmetric Lorenz map
 \begin{equation}
  L(w)=
  \begin{cases}
     2(1-\varepsilon)w+\varepsilon & \text{if }  0 \leq w < \frac{1}{2} \\
     2(1-\varepsilon)w+\varepsilon-1       & \text{if } \frac{1}{2} \leq w \leq 1
    \end{cases} \label{lorenz}
  \end{equation}
 sketched on figure \ref{hl}. Parry \cite{parry1979lorenz} gave a construction for the attractor of this map, which we recall briefly.

 We say that a Lorenz map is renormalizable, if there exists $\ell,r > 1$ and a proper subinterval $[a,b] \subset [0,1]$ such that $g: [a,b] \to [a,b]$ defined as
 \[ g(x) = \left\{
   \begin{array}{l l}
     f^{\ell}(x) & \quad \text{if $x \in [a,\frac{1}{2})$}\\
     f^{r}(x) & \quad \text{if $x \in (\frac{1}{2},b]$}
   \end{array} \right.\]
 is also a Lorenz map. A renormalization $g=(f^{\ell},f^r)$ is minimal, if for any other renormalization $g'=(f^{\ell'},f^{r'})$ of $f$ we have $\ell' \geq \ell$ and $r' \geq r$. Let us denote by $\mathcal{R}f$ the minimal renormalization of $f$. We say that $f$ is $n$ times renormalizable, if $\mathcal{R}^kf$ is renormalizable for $0 \leq k < n$, but $\mathcal{R}^nf$ is not renormalizable. Using this terminology, the result of Parry can be stated in the following way:
 \begin{theo}[Parry]
 Let $\sqrt[2^{n+1}]{2} < 2(1-\varepsilon) < \sqrt[2^n]{2}$. The map $L$ defined by \eqref{lorenz} is $n$ times renormalizable with $\mathcal{R}^kL=L^{2^{k}}|_{\mathcal{J}_k}$ for $0 \leq k \leq n$, and the renormalizalion intervals form a nested sequence around $\frac{1}{2}$:
 \[
 \frac{1}{2} \in \mathcal{J}_k \subset \mathcal{J}_{k-1} \subset \dots \subset \mathcal{J}_1 \subset \mathcal{J}_0=[0,1].
 \]
 \end{theo}

 It is easy to see that if a map is renormalizable, it cannot by mixing. Conversely, Glendinning and Sparrow \cite{glendinning1993prime} states that if a Lorenz map is not renormalizable, then it admits a unique mixing acim with support $[0,1]$. By this and the result of Parry, we see that the map $L$ has an ergodic acim on its attractor which is the union of $2^n$ mixing components if $\sqrt[2^{n+1}]{2} < 2(1-\varepsilon) < \sqrt[2^n]{2}$. We remark that there exists an explicit, albeit involved formula for the invariant density $\bar{\ell}$ of the Lorenz map, see \cite{gora2009invariant}.

 Thus the map $H$ also has an ergodic acim supported on a set of intervals, which is mixing if $\varepsilon < 1-\frac{\sqrt{2}}{2}$ and not mixing if $1-\frac{\sqrt{2}}{2} \leq \varepsilon$.\footnote{There is a countable set of parameters which are literally not treated by Parry's Theorem, yet, in these cases, the map has a Markov partition and it can be concluded directly that it is non-mixing. For instance, if $\varepsilon=1-\frac{\sqrt{2}}{2}$, there exist three intervals with disjoint interiors $I_1,I_2$ and $I_3$ such that $TI_2=I_1\cup I_3$ while $TI_1=TI_3= I_2$. The cases $\eps=1-\frac{ \sqrt[2^n]{2}}{2}$ have analogous behaviour with longer cycles.} Let us denote this measure by $\mu_{H}$ and its density by $\ell$.

 If we think of the phase space of $F_{\varepsilon,2}$ as the unit square with sides identified, by Observation~\ref{obs1} it follows that the circles parallel to the main diagonal are mapped onto each other according to the law of $H$ and $F_{\varepsilon,2}$ acts as the doubling map in the diagonal direction. Hence the original map $F_{\varepsilon,2}$ also has a unique acim $\mu_F$ in the expanding case, which is mixing if $0 \leq \varepsilon < 1-\frac{\sqrt{2}}{2}$.

%  Let us denote by $M$ the phase space of $F_{\varepsilon,2}$ and by $\widetilde{M}$ the phase space of $G_{\varepsilon,2}$. Let $\pi: M \to \widetilde{M}$ be the projection associated with the factor map:
%  \[
%  \pi: \left\{(x,y), \left(x+\frac{1}{2},y+\frac{1}{2}\right) \mod 1 \right\} \mapsto (x,y).
%  \]
%   We are going to define an acim $\mu_F$ in the system $(M,F_{\varepsilon,2})$ in the following way: let $\tilde{h}(x)$ be the density function of $\mu_G$. Then let the density function of $\mu_F$ be
% \[
% h(y)=\frac{1}{2}\tilde{h}(\pi y).
% \]
% It is easy to see that this is in fact a fixed point of the transfer operator associated with $F_{\varepsilon,2}$, hence an invariant density.
%
% When $0 \leq  \varepsilon < \frac{1}{2}$, the ergodicity of $\mu_F$ is the consequence of the ergodicity of $\mu_G$ for such values of $\varepsilon$, and the fact that $(x,y)$ and $\left (x+\frac{1}{2},y+\frac{1}{2}\right)$ have the same image under the action of $F_{\varepsilon,2}$.

Since $\mu_H$ is not mixing if $1-\frac{\sqrt{2}}{2} \leq \varepsilon$, it is straightforward that $\mu_F$ is not mixing for these values of the coupling parameter either. This concludes our proof of the second statement of the theorem.

 Now for the last statement of the theorem, let $\frac{1}{2} < \varepsilon < 1$. In this case the map $H$ becomes contracting and it is easy to see that the trajectory of every point converges to the origin. So the attractor of $F_{\varepsilon,2}$ will be the diagonal of $\mathbb{T}^2$, on which $F_{\varepsilon,2}$ acts as the doubling map.

 %So the attractor of $F_{\varepsilon,2}$ will be the intersection of %the line $x=y$ and $\mathbb{T}^2$, which is an invariant circle on %$\mathbb{T}$ and $F_{\varepsilon,2}$ acts on it as the doubling map.

 \end{proof}

 Some simulation results are pictured on figure \ref{fig666}. We fixed the value $\varepsilon=\frac{1}{3}$. In the simulation depicted on figure \ref{fig1}, we took 1000 uniformly distributed points on the unit torus and plotted the last 1500 elements of the 2000 long trajectory with respect to the factor mapping $G_{1/3,2}$. We can see that the map acts as the doubling map in the direction of the first coordinate and as a map conjugate to a Lorenz map in the direction of the second coordinate. Using the results of Parry we can calculate that the attractor is the union of two mixing components, one is the central strip, the other is the union of the remaining two strips. We get the attractor of our original system by tiling the plane and cutting out the relevant region as illustrated on figure \ref{fig2}. On figure \ref{fig3} we took 1000 uniformly distributed points on the unit torus and plotted the last 1500 elements of the 2000 long trajectory with respect to the mapping $F_{1/3,2}$.
 \begin{figure}
         \centering
         \begin{subfigure}[b]{0.3\textwidth}
                 \centering
                 \includegraphics[width=\textwidth,height=0.2\textheight]{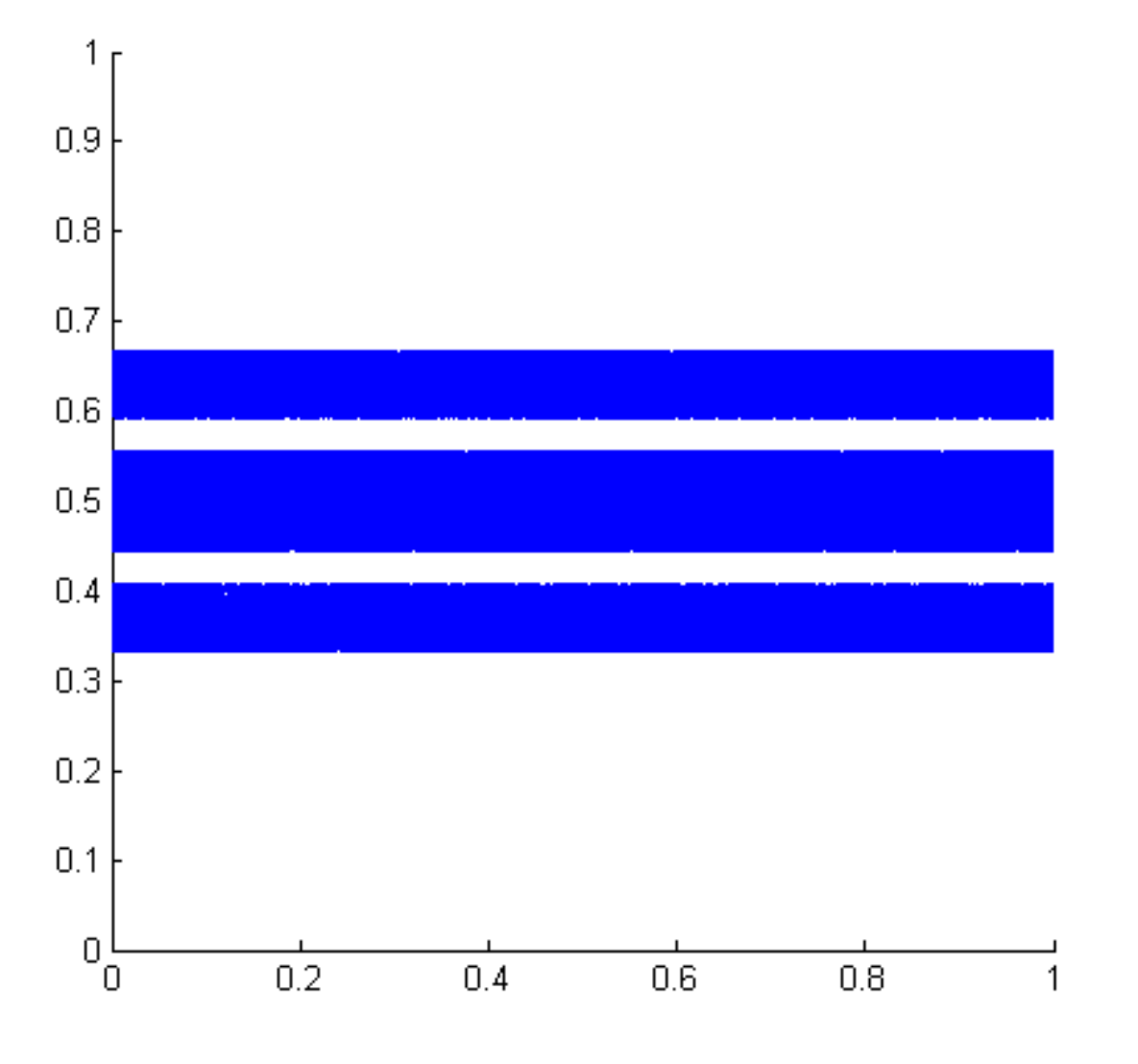}
                 \caption{Simulation of the system with dynamics $G_{1/3,2}$.} \label{fig1}
         \end{subfigure}
         \quad
         \begin{subfigure}[b]{0.3\textwidth}
                 \centering
                 \includegraphics[width=1\textwidth, height=0.2\textheight]{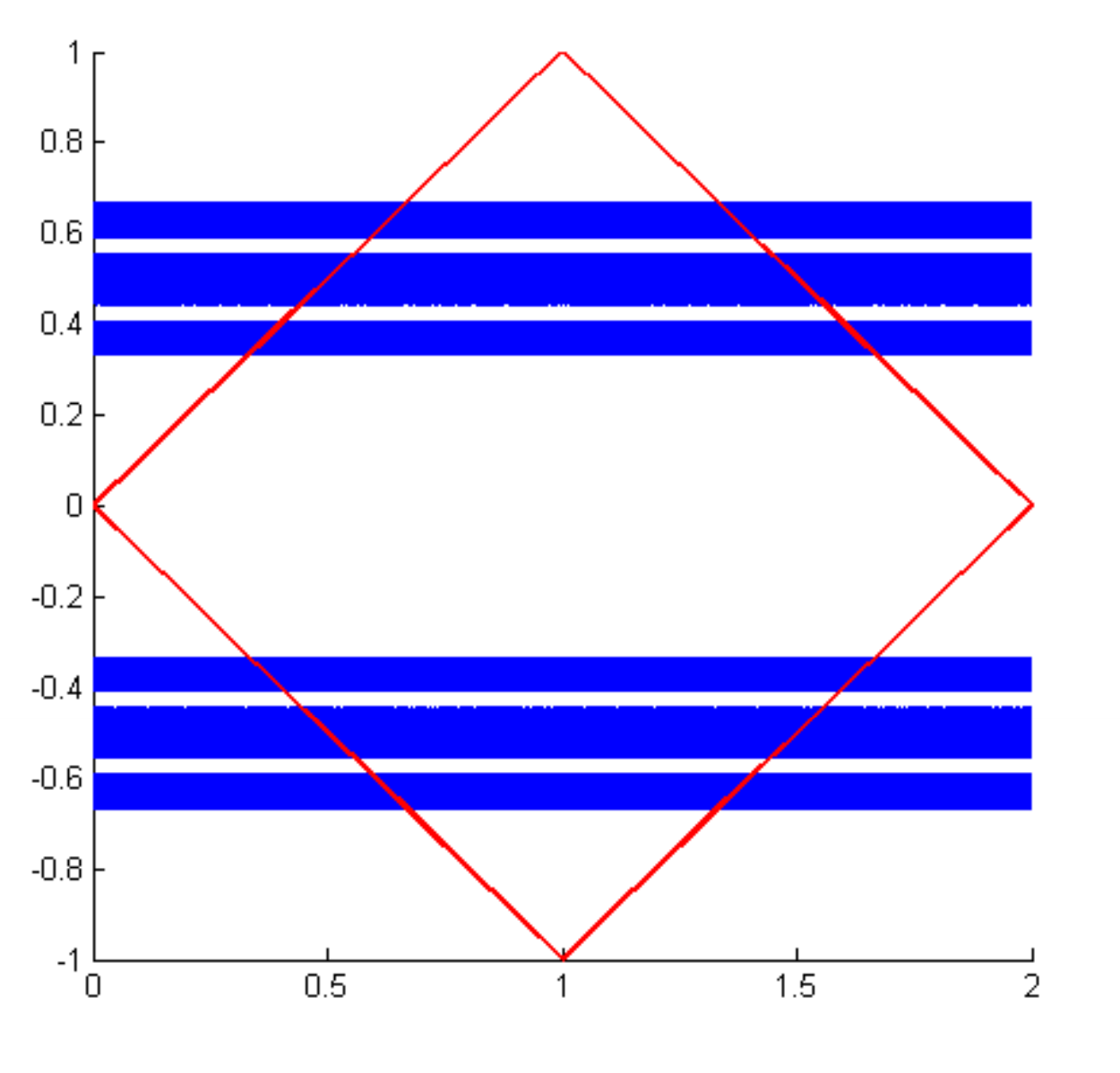}
                 \caption{Tiling of $[0, 2] \times [-1, 1]$ with the phase space of $G_{1/3,2}$.} \label{fig2}
         \end{subfigure}
         \quad
         \begin{subfigure}[b]{0.3\textwidth}
                 \centering
                 \includegraphics[width=\textwidth, height=0.2\textheight]{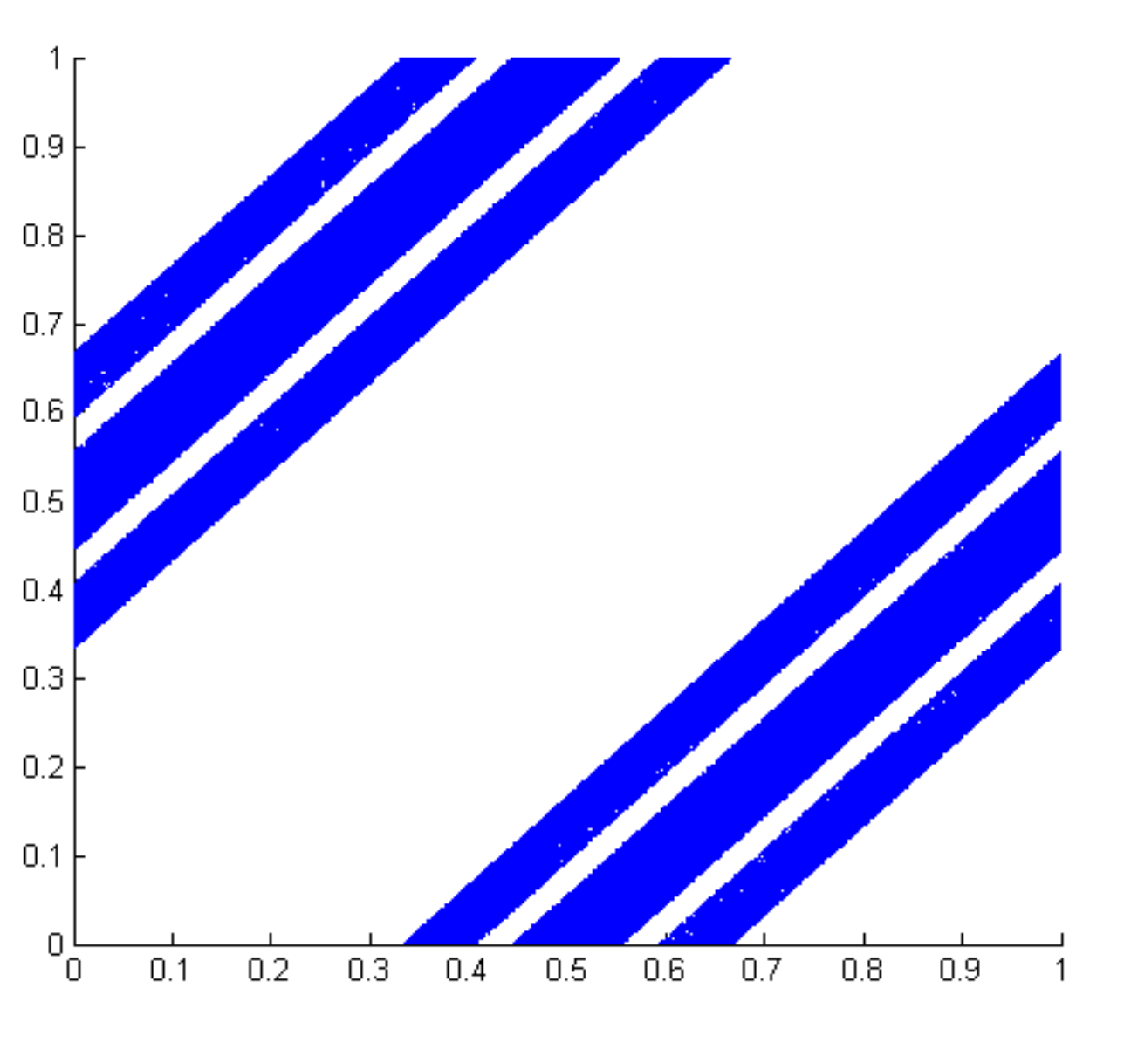}
                 \caption{Simulation of the system with dynamics $F_{1/3,2}$.} \label{fig3}
         \end{subfigure}
         \caption{Simulation of the systems corresponding to the coupling of two sites.} \label{fig666}
 \end{figure}

 We remark that an invariant density is supported on these strips, which is constant in the direction of the diagonal of the square and has a more complicated structure in the orthogonal direction, given by the previously mentioned result of G\'ora \cite{gora2009invariant}.

  Let us, finally, recall Discussion~\ref{d2} and provide some further explanation. In terms of our original dynamical system (evolution of the states of two coupled sites) the results above mean that if the coupling strength $\varepsilon$ increases, certain restrictions arise for the \emph{relative} position of the two sites. In particular, as $\eps$ approaches $\frac12$ from below, the $v$ coordinate is restricted to
 smaller and smaller regions about $\frac12$, which means that the two sites occupy almost opposite locations on the circle. If $\sqrt[2^{n+1}]{2} \leq 2(1-\varepsilon) < \sqrt[2^n]{2}$, the corresponding Lorenz map has $K=K(\varepsilon)=2^n$ mixing components. Accordingly, in course of $K$ iterations, the relative position of the sites keep changing, and then they jump back to an almost opposite position. On the other hand, once $\varepsilon$ becomes larger than $\frac{1}{2}$, the two sites synchronize asymptotically, both sites converge to one point evolving according to the doubling map.

 \subsection{\boldmath$N=3$}

 In this subsection, we are going to prove Theorem~\ref{t3} and provide the background for Discussion~\ref{d3}.
 %which states that for sufficiently weak coupling the system has a unique mixing absolutely continuous invariant measure, if the coupling parameter surpasses a certain value six %ergodic components arise, and if the coupling is sufficiently strong, all trajectories converge to union of three circles. We conclude by giving some comments about the breaking %of ergodicity and the case of strong coupling in terms of the coupled map system.

 \begin{proof}[Proof of Theorem \ref{t3}.] In case of three sites, \eqref{mainsing} will take the form
   \begin{align*} F_{\varepsilon,3}(x,y,z)=\bigg(&2x+\frac{2\varepsilon}{3}(g(y-x)+g(z-x)),2y+\frac{2\varepsilon}{3}(g(x-y)+g(z-y)), \\ & 2z+\frac{2\varepsilon}{3}(g(x-z)+g(y-z))\bigg) \quad  \text{mod } 1, \quad x,y,z \in \mathbb{T}.
   \end{align*}

  Let us consider again the factor dynamical system defined by
  \begin{align*}
  G_{\varepsilon,3}(w,u,v)=(G_{\varepsilon,3}^1(w),G_{\varepsilon,3}^2(u,v)),
  \end{align*}

  The coordinates $w,u,v$ correspond to $x+y+z, x-y, y-z$ respectively, and we note that $(x,y,z)$, $\left(x+\frac{1}{3}, y+\frac{1}{3}, z+\frac{1}{3}\right)$ and $\left(x+\frac{2}{3}, y+\frac{2}{3}, z+\frac{2}{3}\right)$ share the same $(w,u,v)$ coordinates.

%  and the projection
%  \[
%  \pi: \left\{ (x,y,z), \left(x+\frac{1}{3}, y+\frac{1}{3}, z+\frac{1}{3}\right), \left(x+\frac{2}{3}, y+\frac{2}{3}, z+\frac{2}{3}\right)\right\} \mapsto (x,y,z)
%  \]
%   is a three-fold covering of $\mathbb{T}^3$. (Additions are always considered mod 1.)

The phase space of $G_{\varepsilon,3}$ is $\mathbb{T}^3$, but we have to use a careful representation to get a geometrically precise picture. Regarding the original map $F_{\varepsilon,3}$, we most naturally think of the phase space $\mathbb{T}^3$ with coordinates $x,y$ and $z$ as the unit cube of $\mathbb{R}^3$ in the canonical basis. Now we want our new coordinates to be $x+y+z$, $x-y$ and $y-z$. This means our new basis should be $e_1=\frac{1}{3}(1,1,1)$, $e_2=\frac{1}{3}(2,-1,-1)$ and $e_3=\frac{1}{3}(1,1,-2)$. Note that $e_1 $ is perpendicular to the plane of $e_2$ and $e_3$ and the angle of $e_2$ and $e_3$ is $60^{\circ}$.

In conclusion, the phase space of $G_{\varepsilon,3}$ should be represented as a prism with the unit square defined by $e_2$ and $e_3$ as its base (which happens to be a rhombus with angles $60^{\circ}$ and $120^{\circ}$). The direction of $e_1$ needs no special analysis, since $G_{\varepsilon,3}^1$ is the doubling map, so we are going to focus on the map
\begin{align*}
G_{\varepsilon,3}^2(u,v)=\left(2u+\frac{2\varepsilon}{3}(g(v)-g(u+v)-2g(u)), 2v+\frac{2\varepsilon}{3}(g(u)-g(u+v)-2g(v)) \right) &\quad \text{mod 1}, \\
 &\quad (u,v) \in \mathbb{T}^2.
\end{align*}
acting on the rhombus pictured on \ref{figr}.
   \begin{figure}[h!]
   \centering
    \begin{subfigure}[b]{0.45\textwidth}
    \centering
   \includegraphics[scale=0.6]{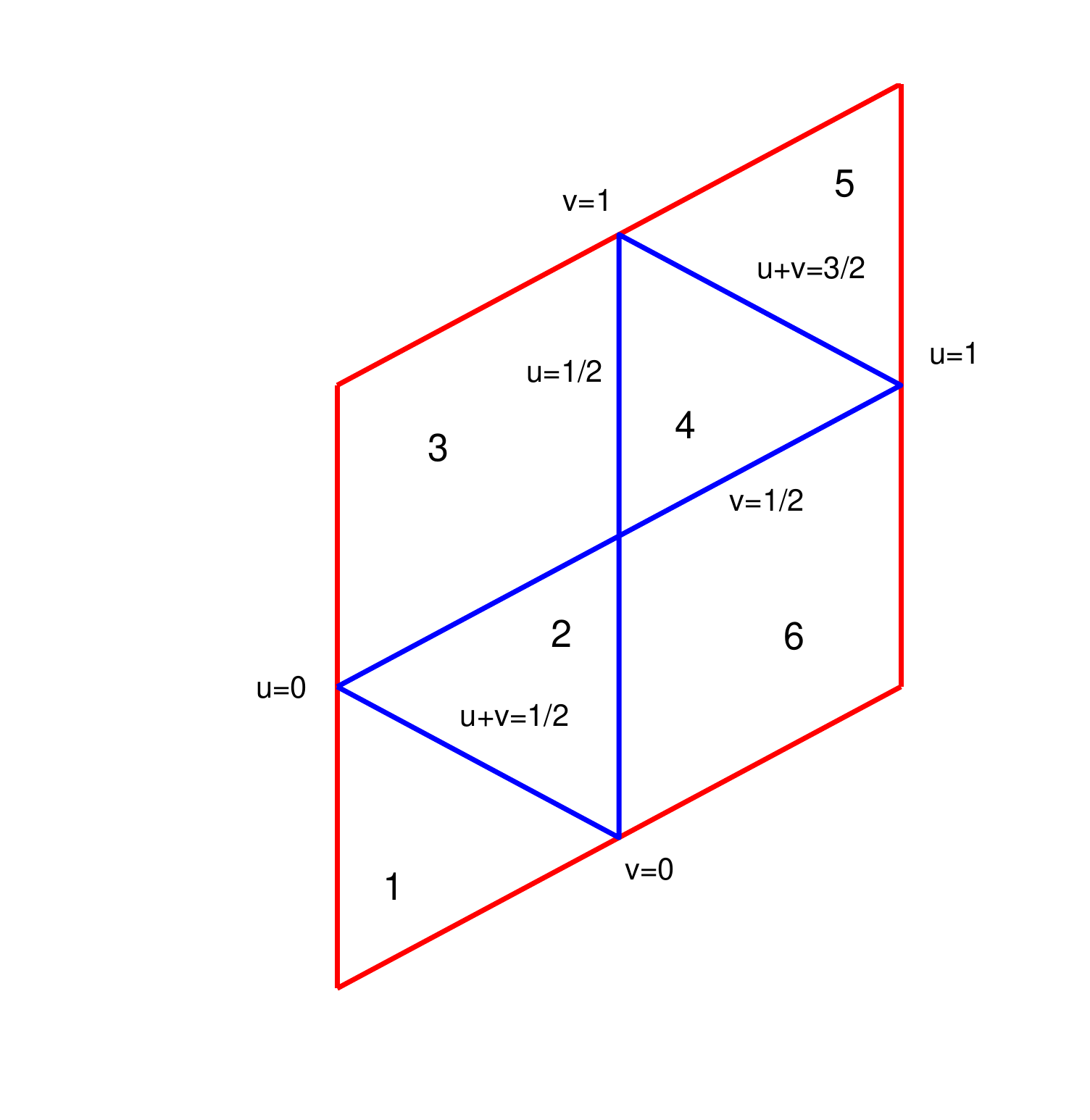}
   \caption{Phase space of $G_{\varepsilon,3}^2$. Blue lines mark the singularities of the map, the map is given by equations \eqref{h1}-\eqref{h6} on each numbered domain.} \label{figr}
   \end{subfigure}
   \qquad
   \begin{subfigure}[b]{0.45\textwidth}
                   \centering
                   \includegraphics[scale=0.65]{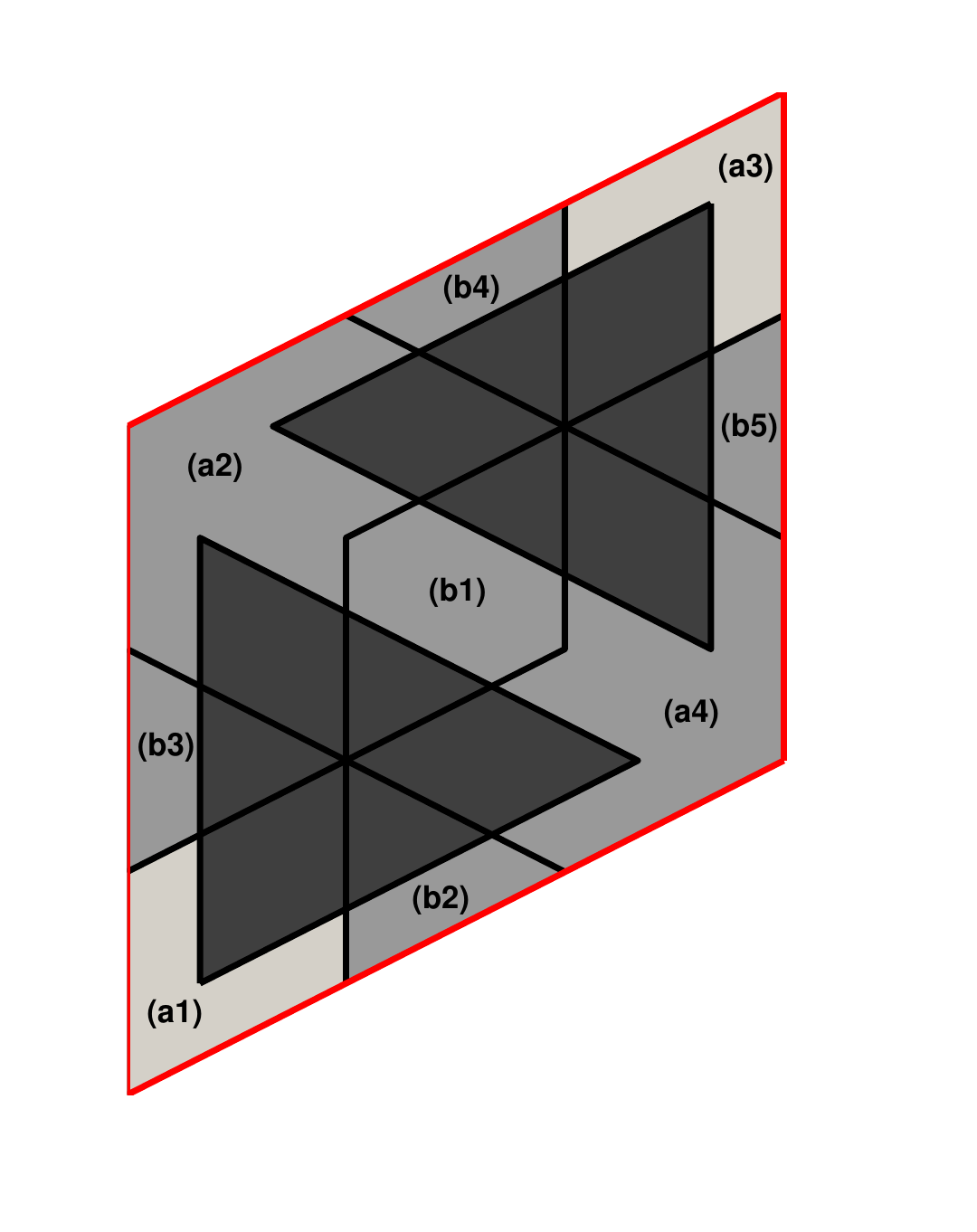}
                   \caption{The image of the six domains from figure \ref{figr} under the action of the map $G_{\varepsilon,3}^2$. Light grey: image of 1,5. Medium grey: image of 3,6. Dark grey: image of 2,4.} \label{figd1}
           \end{subfigure}
   \caption{Phase space of $G_{\varepsilon,3}^2$ and the image of the continuity domains.}
   \end{figure}

 On the six pictured domains (giving three domains of continuity) the map takes the following form:
  \begin{align}
  1. \quad  (u,v) \mapsto & (2(1-\varepsilon)u,2(1-\varepsilon)v) \quad \text{ mod } 1 \label{h1}\\
  2. \quad  (u,v) \mapsto & \left(2(1-\varepsilon)u+\frac{2\varepsilon}{3},2(1-\varepsilon)v+\frac{2\varepsilon}{3}\right) \quad \text{ mod } 1 \\
  3. \quad  (u,v) \mapsto & (2(1-\varepsilon)u,2(1-\varepsilon)v+2\varepsilon) \quad \text{ mod } 1 \\
  4. \quad  (u,v) \mapsto & \left(2(1-\varepsilon)u+\frac{4\varepsilon}{3},2(1-\varepsilon)v+\frac{4\varepsilon}{3}\right) \quad \text{ mod } 1 \\
  5. \quad  (u,v) \mapsto & (2(1-\varepsilon)u+2\varepsilon,2(1-\varepsilon)v+2\varepsilon) \quad \text{ mod } 1 \\
  6. \quad  (u,v) \mapsto & (2(1-\varepsilon)u+2\varepsilon,2(1-\varepsilon)v) \quad \text{ mod } 1 \label{h6}
  \end{align}
   First let $0 \leq \varepsilon < \frac{1}{2}$. In this case the map is expanding, since the Jacobian is $2(1-\varepsilon)I$ on each domain of continuity, where $I$ is the $2 \times 2$ identity matrix.

  The image of domains 1-6 is depicted on figure \ref{figd1}. We can see that the union of domains $(a1), (a2),$ $(a3)$ and $(a4)$ cannot be a part of the Milnor attractor, since the only way a point can get in them is if it was already in them, but points will eventually leave these domains due to expansion. The preimage of the union of domains $(b1)-(b5)$ is contained in itself and the union of $(a1)-(a4)$, hence the union of $(b1)-(b5)$ cannot be part of the attractor either. So the attractor is contained in the images of domains $2$ and $4$, the union of two triangles. But the three corners of each triangle as defined on figure \ref{figd2}, will only have a preimage in the union of  $(a1)-(a4)$ and the union of $(b1)-(b5)$, hence these corners will not be a part of the attractor either. Hence the attractor is contained in the union of two hexagons, pictured on figure \ref{figd2}.

  We now move on to proving the first statement of the theorem. First we note that by expansion, we may again rely on the results of \cite{thomine2010spectral} and \cite{saussol2000absolutely}: for the uniqueness of the absolutely continuous invariant measure, it is enough to show that the map is locally eventually onto (l.e.o. for short). Note, furthermore, that it is enough to show that the factor map $G$ restricted to its attractor is l.e.o., because then by Observation~\ref{obs1} it follows that the original map $F$ restricted to its attractor is also l.e.o. (note that this conclusion would hold for any $N \in \mathbb{N}$). Furthermore, it will suffice to prove that the map $G_{\varepsilon,3}^2$ restricted to the two hexagons is l.e.o. if $\varepsilon < 1-\frac{\sqrt{2}}{2}$ (because the doubling map $G_{\varepsilon,3}^1$ also has this property).

  To prove this, let us take a small set on the attractor. First we prove that this set will have an image which intersects the diagonal $u=v$. To see this, notice that if two or more singularity lines cross this set, then the set already crosses the diagonal. Hence without loss of generality  we may assume that either zero or one singularity line crosses this set. If no singularity crosses the set, then its area will grow with a factor of $[2(1-\varepsilon)]^2 >1$. If one singularity crosses the set, then the image will have two pieces and one will have an area larger than $\frac{[2(1-\varepsilon)]^2}{2}$ times the area of the original set, which factor is also larger then 1 if $\varepsilon < 1-\frac{\sqrt{2}}{2}$. Hence the image of the set grows, but that cannot go on forever, and eventually it will have to be cut by two singularities, but then it necessarily crosses the diagonal.

  The dynamics restricted to $u=v$ is
    \[
    H(u)=2u-\frac{2\varepsilon}{3}(g(u)+g(2u)) \mod 1.
    \]
 Recall that $\sigma=\frac{2\varepsilon}{3}(2-\varepsilon)$, as defined in the caption of figure~\ref{figd2}.  The intersection of the two hexagons and the diagonal $u=v$ is the subset of the diagonal such that
  \[
  u \in \left[\frac{\sigma}{2}, \frac{1}{2}-\frac{\varepsilon}{6} \right] \cup \left[\frac{1}{2}+\frac{\varepsilon}{6}, 1-\frac{\sigma}{2} \right],
  \]

  Let us denote this subset of the diagonal by $\kappa$. Fernandez \cite{fernandez2014breaking} proved that $H$ restricted to $\kappa$ is locally eventually onto, hence our set will have an image which contains a small neighborhood of $\left(\frac{1}{3},\frac{1}{3}\right)$, and it is easy to see that this small neighborhood will have an image which covers the intersection of domain 2 and the attractor. The image of this set is the upper hexagon, and the image of the upper hexagon covers the lower hexagon, hence we are done.

  Now we move on to breaking of ergodicity, that is, we are going to show that if $\varepsilon$ is above a certain threshold value, the Milnor attractor
   consists of multiple invariant sets, hence it supports multiple ergodic invariant measures.

 Let us include the lines
  \[
  v=u, \quad v=-\frac{1}{2}u+1, \quad  v=-2u+1, \quad v=-2u+2, \quad v=-\frac{1}{2}u+\frac{1}{2} \quad \text{and} \quad v=1-u
  \] on our picture, see the orange and the green lines on figure \ref{figd3}. The orange line separates the two hexagons, and
  the green lines partition both hexagons into six domains $((Ix),(IIx),\dots,(VIx)$ for $x=a,b)$ which can be paired up to get the domains $I,II,\dots,VI$.

  We are going to show that each such domain is invariant if $\varepsilon$ is large enough. We will give a detailed proof for domain $I$, the proof is analogous for the remaining five domains.

  \begin{figure}[h!]
          \centering
          \begin{subfigure}[b]{0.45\textwidth}
                  \centering
                  \includegraphics[scale=0.6]{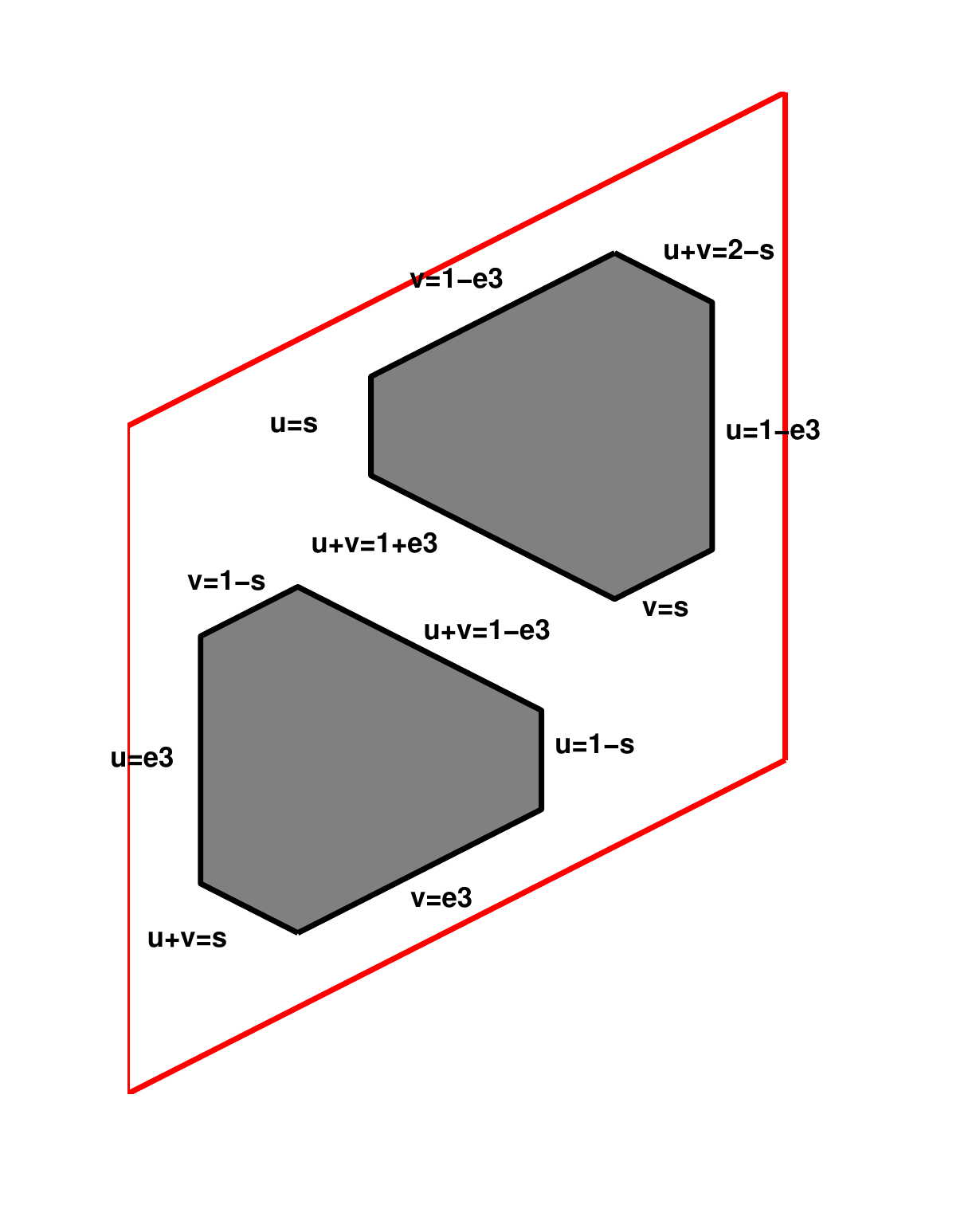}
                  \caption{If $\varepsilon \leq \frac{1}{3}$, the Milnor attractor is exactly the union of these two hexagons, otherwise it is contained in them.  We used the notation $\frac{2\varepsilon}{3}(2-\varepsilon)=\sigma$.} \label{figd2}
          \end{subfigure}
          \qquad
          \begin{subfigure}[b]{0.45\textwidth}
                  \centering
                  \includegraphics[scale=0.6]{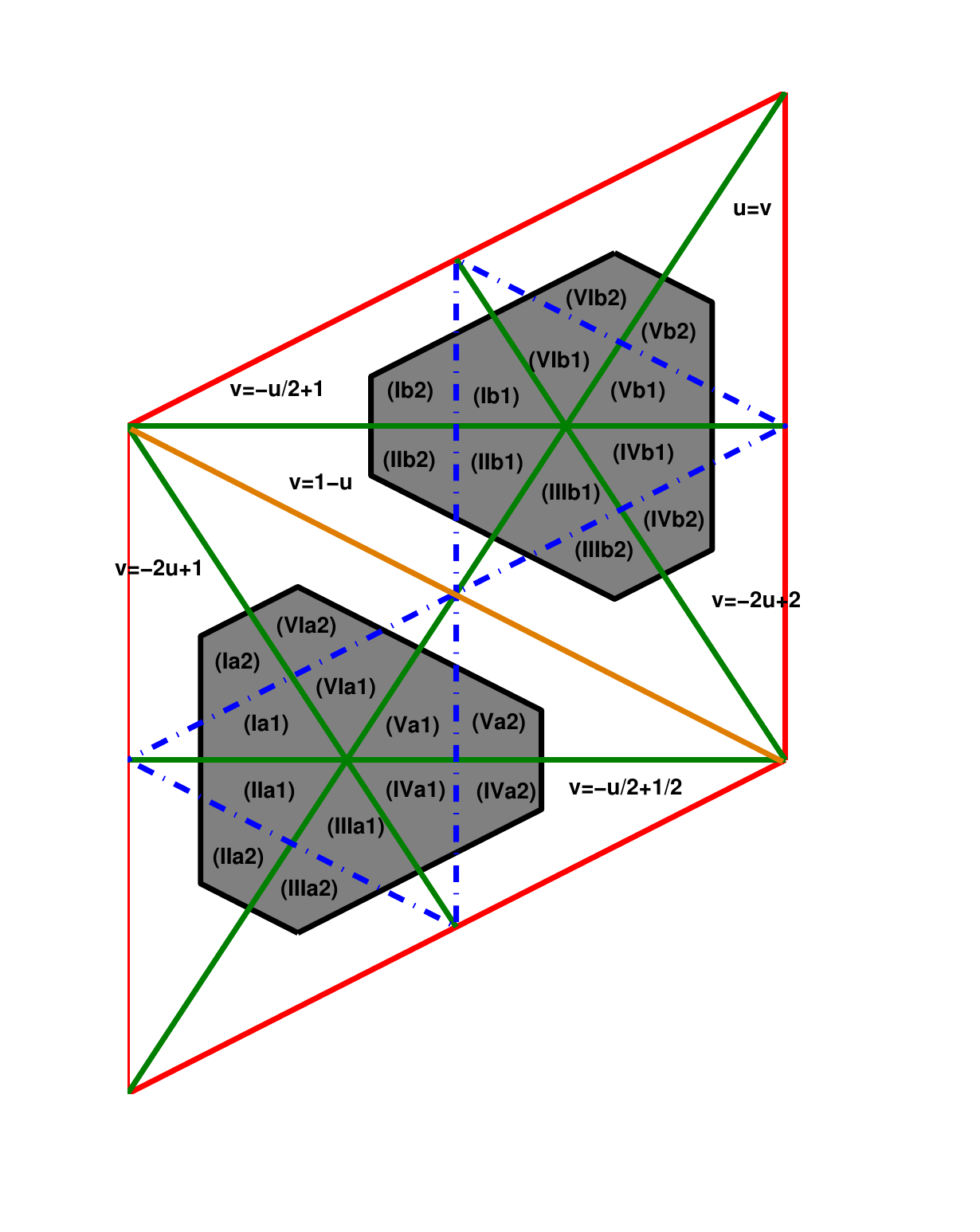}
                  \caption{We get the domains $I,II,\dots,VI$ the following way: the union of $(Wx1)$ and $(Wx2)$ is $(Wx)$ for $x=a,b$, $W=I,II,\dots,VI$, and the union of $(Wa)$ and $(Wb)$ is $W$ for $W=I,II,\dots,VI$.} \label{figd3}
          \end{subfigure}
  \caption{Structure of the Milnor attractor.}
  \end{figure}

   Domain $(Ia)$ consists of a quadrilateral $(Ia1)$ in continuity domain $2$ and a quadrilateral $(Ia2)$ in continunity domain $3$. Similarly, domain $(Ib)$ consists of a quadrilateral $(Ib1)$ in continuity domain $4$ and a quadrilateral $(Ib2)$ in continunity domain $3$.

   Let us first deal with $(Ia1)$ and $(Ib1)$. The four sides of $(Ia1)$, namely $v=-\frac{1}{2}u+\frac{1}{2}$, $v=-2u+1$, $v=\frac{1}{2}$ and $u=\frac{\varepsilon}{3}$ map under the action of $G_{\varepsilon,3}^2$ in continuity domain 2, respectively, into the lines $v=-\frac{1}{2}u+1$, $v=-2u+2$, $v=1-\frac{\varepsilon}{3}$ and $u=\sigma$. So we see that the image of $(Ia1)$ is precisely $(Ib)$. In a similar fashion, the four sides of $(Ib1)$, namely $v=-\frac{1}{2}u+1$, $v=-2u+2$, $u=\frac{1}{2}$ and $v=1-\frac{\varepsilon}{3}$ map under the action of $G_{\varepsilon,3}^2$ in continuity domain 4, respectively, into the lines $v=-\frac{1}{2}u+\frac{1}{2}$, $v=-2u+1$, $u=\frac{\varepsilon}{3}$ and $u=1-\sigma$. Hence the image of $(Ib1)$ is precisely $(Ia)$.

   Now let us deal with $(Ia2)$ and $(Ib2)$. If these are contained in $(Ia)$ and $(Ib)$ respectively, the domain $I$ is invariant. Let us look at $(Ia2)$ first. Under the action of our map on continuity domain $3$, the $u=\frac{\varepsilon}{3}$  border will move towards the interior of domain $I$ and the border $v=-2u+1$ is invariant, hence the image of $(Ia2)$ can only leave $I$ if at least one of its two corners $\left( \frac{\varepsilon}{3},\frac{1}{2}\right)$ or $\left( \frac{1}{4},\frac{1}{2}\right)$ pass the line $v=-\frac{1}{2}u+\frac{1}{2}$. The image of $\left( \frac{\varepsilon}{3},\frac{1}{2}\right)$ is $\left( 2(1-\varepsilon)\frac{\varepsilon}{3},\varepsilon \right)$, and the picture of $\left( \frac{1}{4},\frac{1}{2}\right)$ is $\left( \frac{1}{2}(1-\varepsilon),\varepsilon\right)$. Hence the image of $(Ia2)$ is in $I$ if and only if
  \begin{align*}
  &-\frac{1}{2}\cdot \frac{1}{2}(1-\varepsilon)+\frac{1}{2} \leq \varepsilon \Leftrightarrow \frac{1}{3} \leq \varepsilon, \\
  &-\frac{\varepsilon}{3}(1-\varepsilon)+\frac{1}{2} \leq \varepsilon \Leftrightarrow \frac{4-\sqrt{10}}{2} \leq \varepsilon.
  \end{align*}
  The argument is very much the same for domain $(Ib2)$. Under the action of our map on continuity domain $3$, the $v=1-\frac{\varepsilon}{3}$  border will move towards the interior of domain $I$ and the border $v=-\frac{1}{2}u+1$ is invariant, hence the image of $(Ib2)$ can only leave $I$ if one of its two corners $\left( \frac{1}{2},1-\frac{\varepsilon}{3}\right)$ or $\left( \frac{1}{2},\frac{3}{4}\right)$ pass the line $v=-2u+2$. The image of $\left( \frac{1}{2},1-\frac{\varepsilon}{3}\right)$ is $\left( 1-\varepsilon,2(1-\varepsilon)\left(1-\frac{\varepsilon}{3}\right)+2\varepsilon-1 \right)$, and the image of $\left( \frac{1}{2},\frac{3}{4}\right)$ is $\left( 1-\varepsilon, \frac{1}{2}(1+\varepsilon)\right)$. Hence the image of $(Ib2)$ is in $I$ if and only if
  \begin{align*}
  &1-\varepsilon \leq 1-\frac{1}{4}(1+\varepsilon) \Leftrightarrow \frac{1}{3} \leq \varepsilon, \\
  &1-\varepsilon \leq 1-(1-\varepsilon)\left(1-\frac{\varepsilon}{3}\right)-\varepsilon+\frac{1}{2} \Leftrightarrow \frac{4-\sqrt{10}}{2} \leq \varepsilon.
  \end{align*}
  In conclusion is $I$ is invariant if and only if $ \frac{4-\sqrt{10}}{2} \leq \varepsilon $.

  The proof that domains $II-VI$ are invariant if $ \frac{4-\sqrt{10}}{2} \leq \varepsilon $ is completely analogous. Hence at least six invariant sets exist if $ \frac{4-\sqrt{10}}{2} \leq \varepsilon $.

   So if $\frac{4-\sqrt{10}}{2} \leq \varepsilon $, the system of $G_{\varepsilon,3}^2$ has multiple ergodic components, and then the system of $G_{\varepsilon,3}$ must also have multiple ergodic components. In conclusion $F_{\varepsilon,3}$  must have multiple (at least six) ergodic components if $ \frac{4-\sqrt{10}}{2} \leq \varepsilon $. Based on simulations we conjecture that if $0 < \varepsilon < \frac{4-\sqrt{10}}{2}$, the system is ergodic and if $\frac{4-\sqrt{10}}{2} \leq \varepsilon < \frac{1}{2}$ the system has precisely six ergodic components.

  Now let us move on to the case $\frac{1}{2} < \varepsilon < 1$, concerning the last statement of our theorem. In this case the map $G_{\varepsilon,3}^2$ becomes contracting, and points from domains $1,3,5$ and $6$ converge to the origin. Points from domains $2$ and $4$ either end up in the origin or on the periodic orbit of length two consisting of $\left(\frac{1}{3},\frac{1}{3}\right)$ and $\left(\frac{2}{3},\frac{2}{3}\right)$. Hence the attractor of $G_{\varepsilon,3}$ is the union of an invariant circle and two circles which are the images of each other. Now consider the phase space of $G_{\varepsilon,3}$  as a three dimensional domain, a prism with rhomboidal base in the plane $x+y+z=0$. If we tile $\mathbb{R}^3$ by copies of this three dimensional domain, and take the intersection of the unit cube with this tiling, we see that the attractor in the unit cube (the phase space of $F_{\varepsilon,3}$) is the union of three circles. The invariant circle is $(x,x,x)$, $x \in [0, 1]$, the second circle can be parametrized as
  \begin{align*}
  \left(x,x+\frac{2}{3},x+\frac{1}{3}\right) &\text{ if } x \in \left[0, \frac{1}{3}\right], \\
  \left(x,x-\frac{1}{3},x+\frac{1}{3}\right) &\text{ if } x \in \left[\frac{1}{3}, \frac{2}{3}\right], \\
  \left(x,x-\frac{1}{3},x-\frac{2}{3}\right) &\text{ if } x \in \left[ \frac{2}{3},1\right],
  \end{align*}
  and the third one can be parametrized as
  \begin{align*}
  \left(x,x+\frac{1}{3},x+\frac{2}{3}\right) &\text{ if } x \in \left[0, \frac{1}{3}\right], \\
  \left(x,x+\frac{1}{3},x-\frac{1}{3}\right) &\text{ if } x \in \left[\frac{1}{3}, \frac{2}{3}\right], \\
  \left(x,x-\frac{2}{3},x-\frac{1}{3}\right) &\text{ if } x \in \left[ \frac{2}{3},1\right].
  \end{align*}
 By the identifications of the opposite sides of the unit cube, circles are obtained as the union of these intervals. We note that Koiller and Young already obtained this result in \cite{koiller2010coupled}.
 \end{proof}

 \begin{figure}
          \centering
          \begin{subfigure}[b]{0.3\textwidth}
                          \centering
                          \includegraphics[width=\textwidth,height=0.25\textheight]{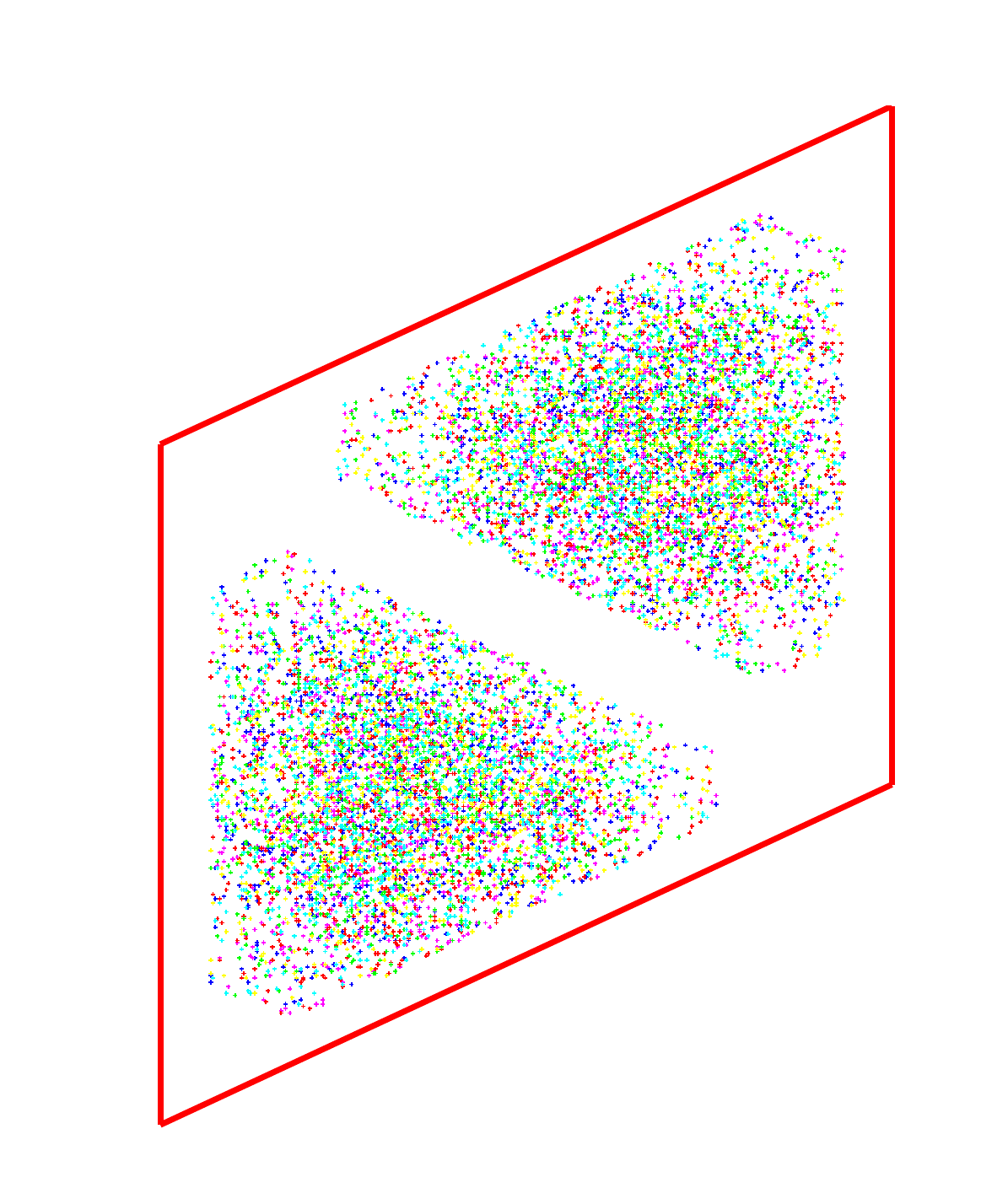}
                          \caption{$\varepsilon=0.2$, Simulation of $G_{\varepsilon,3}^2$.} \label{figc1}
                  \end{subfigure}
                  \begin{subfigure}[b]{0.34\textwidth}
                          \centering
                          \includegraphics[width=\textwidth, height=0.22\textheight]{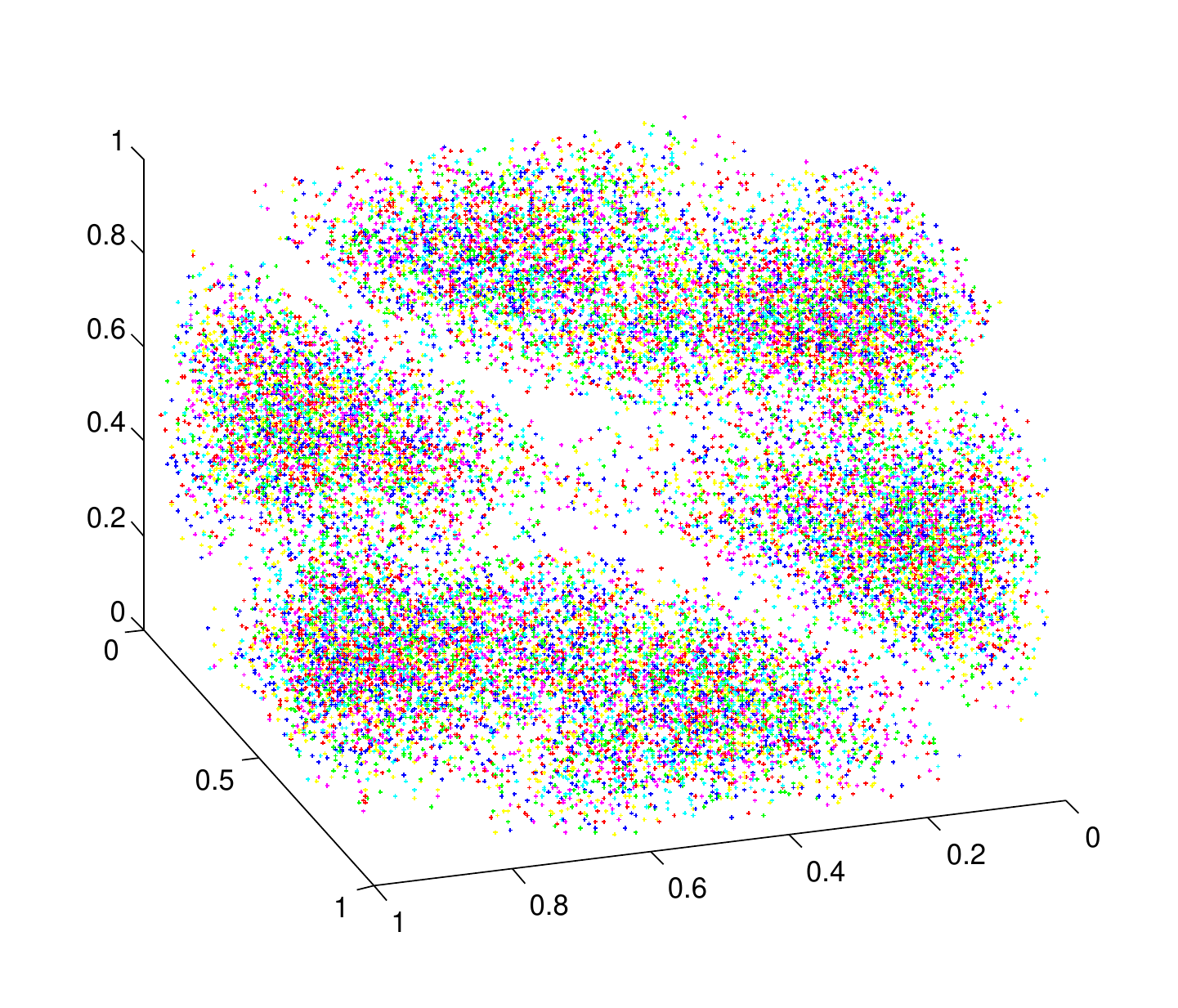}
                          \caption{$\varepsilon=0.2$, Simulation of $F_{\varepsilon,3}$.} \label{figc2}
                  \end{subfigure}
                  \begin{subfigure}[b]{0.34\textwidth}
                          \centering
                          \includegraphics[width=\textwidth, height=0.22\textheight]{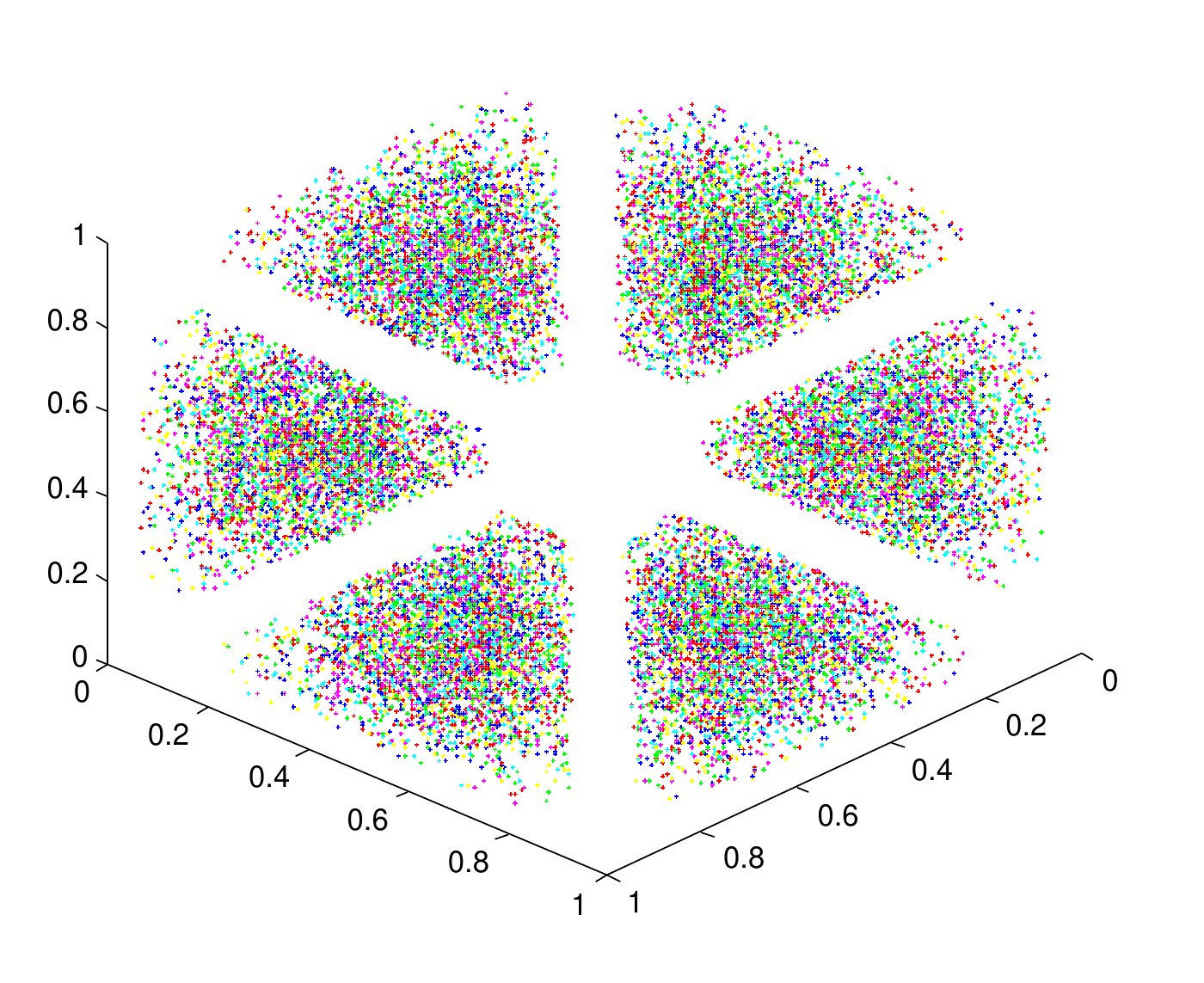}
                          \caption{$\varepsilon=0.2$, Simulation of $F_{\varepsilon,3}$.} \label{figc3}
                  \end{subfigure}
          \begin{subfigure}[b]{0.3\textwidth}
                  \centering
                  \includegraphics[width=\textwidth,height=0.25\textheight]{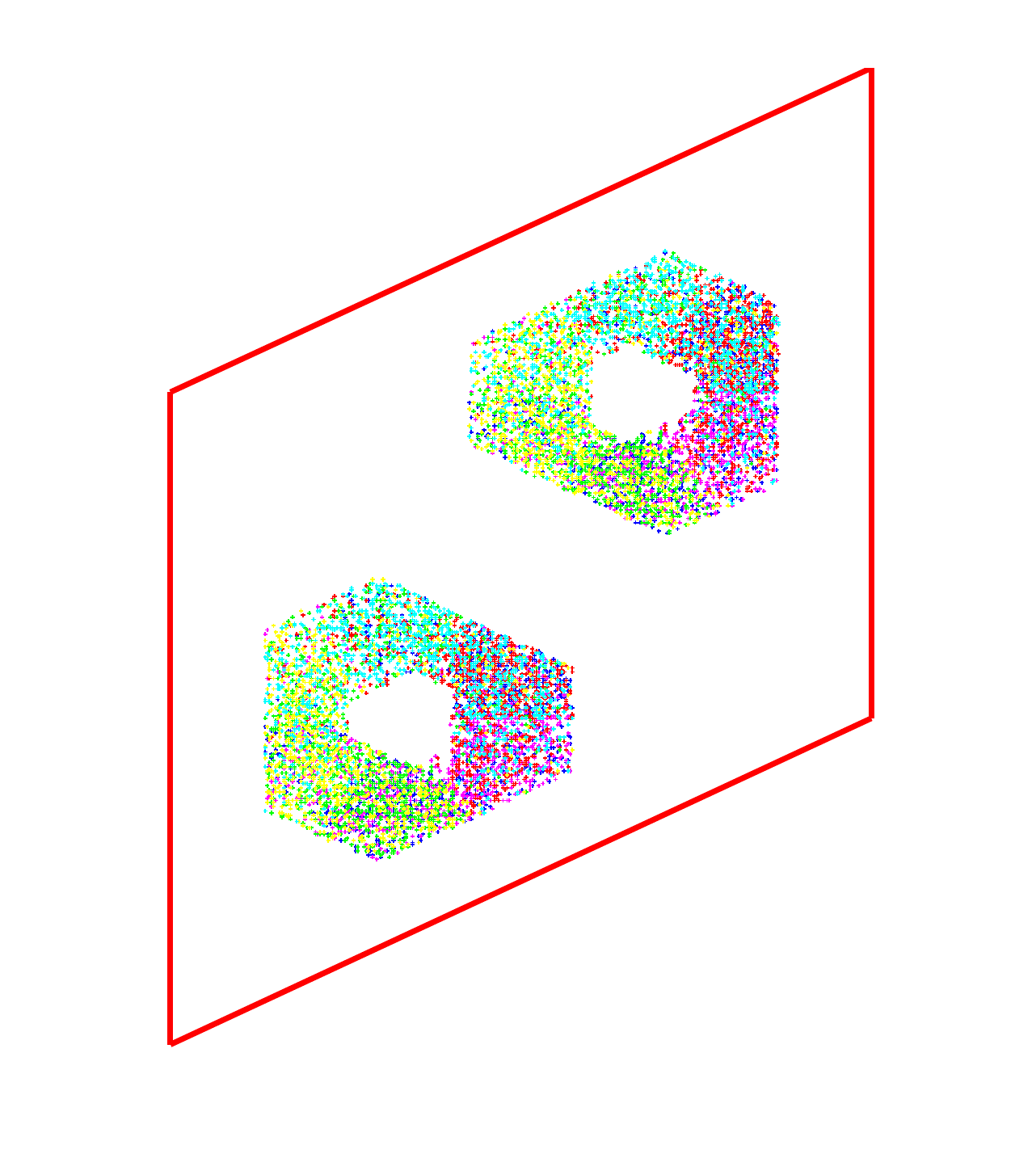}
                  \caption{$\varepsilon=0.4$, Simulation of $G_{\varepsilon,3}^2$.} \label{figa1}
          \end{subfigure}
          \begin{subfigure}[b]{0.34\textwidth}
                  \centering
                  \includegraphics[width=\textwidth, height=0.22\textheight]{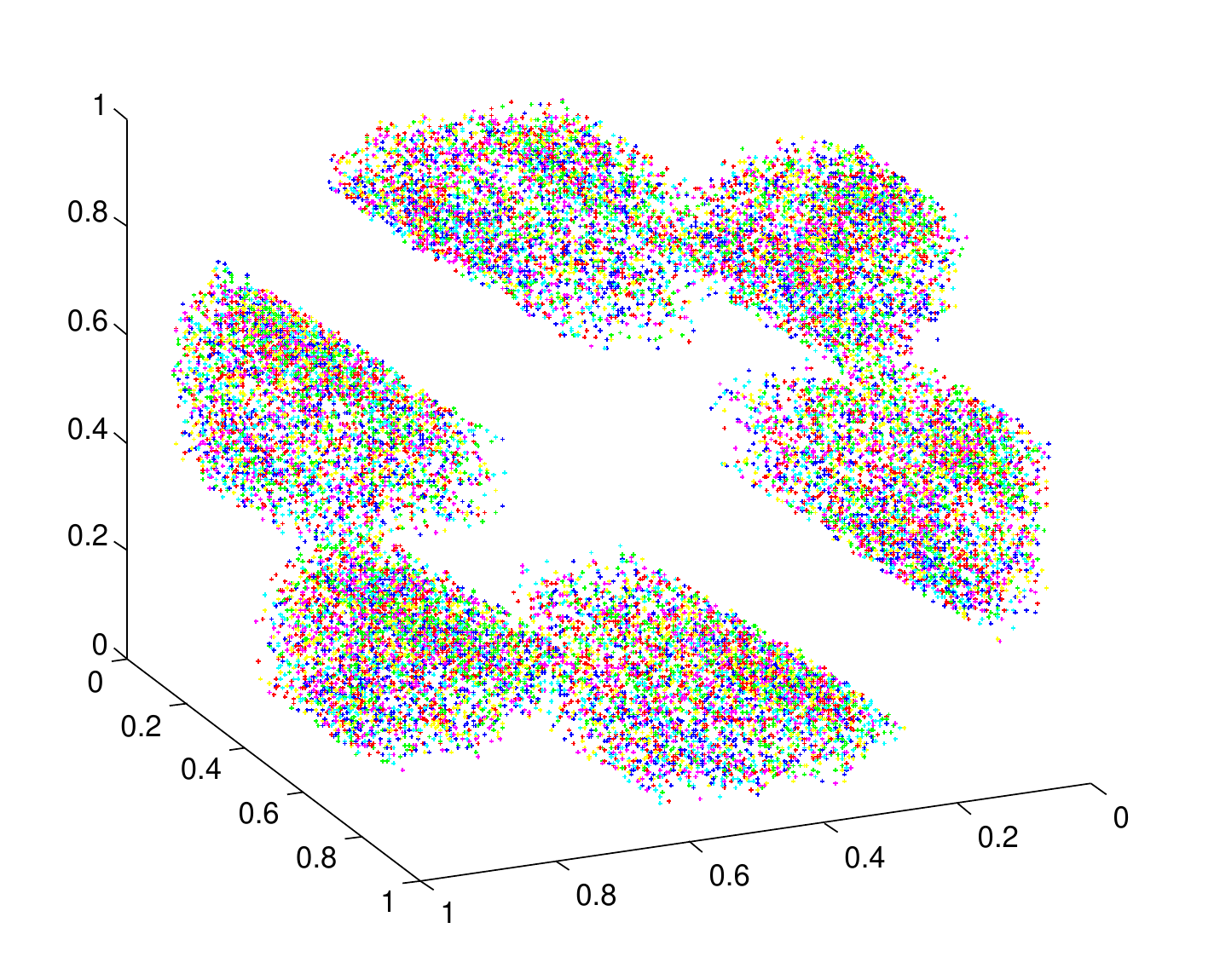}
                  \caption{$\varepsilon=0.4$, Simulation of $F_{\varepsilon,3}$.} \label{figa2}
          \end{subfigure}
          \begin{subfigure}[b]{0.34\textwidth}
                  \centering
                  \includegraphics[width=\textwidth, height=0.22\textheight]{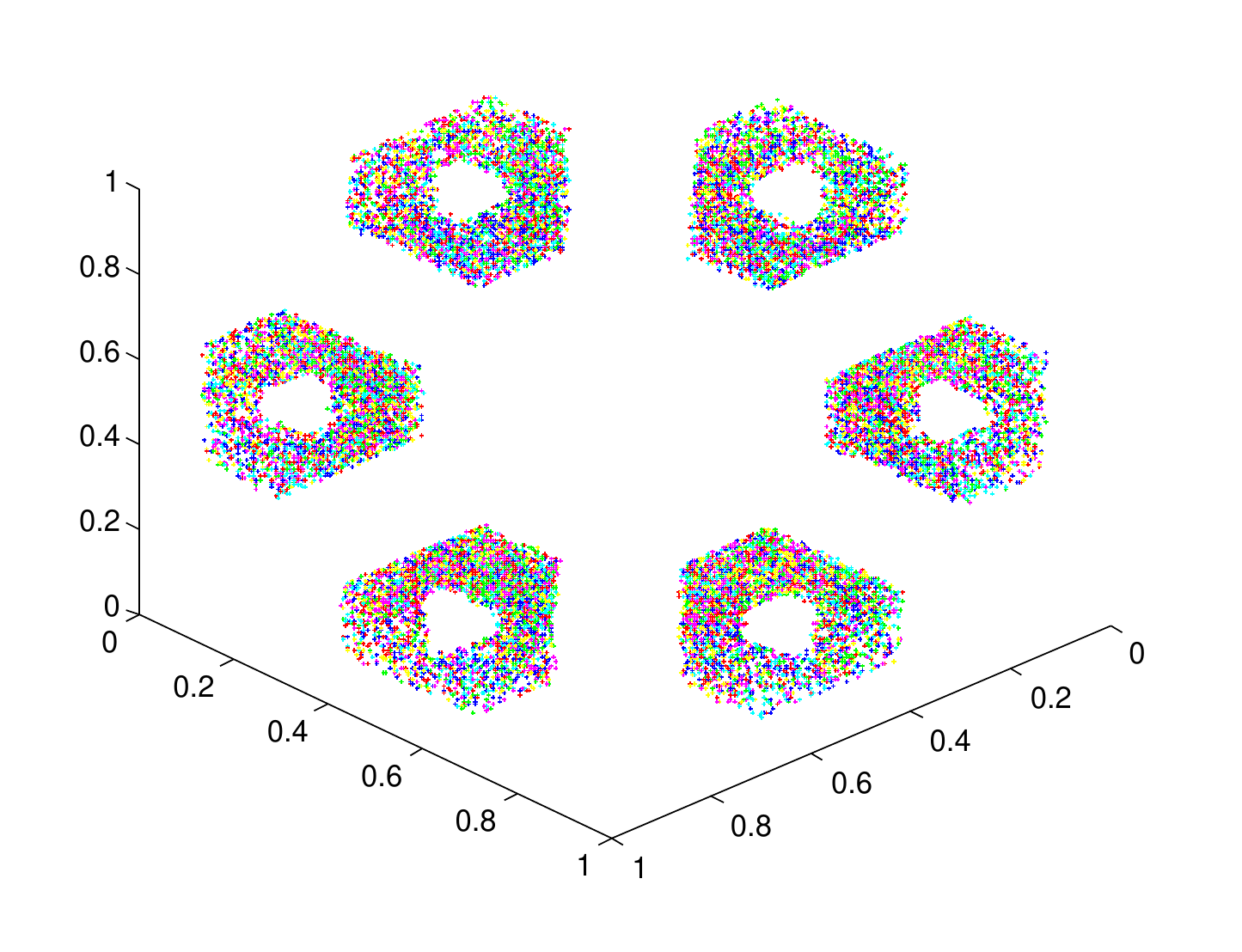}
                  \caption{$\varepsilon=0.4$, Simulation of $F_{\varepsilon,3}$.} \label{figa3}
          \end{subfigure}
          \begin{subfigure}[b]{0.3\textwidth}
                          \centering
                          \includegraphics[width=\textwidth,height=0.25\textheight]{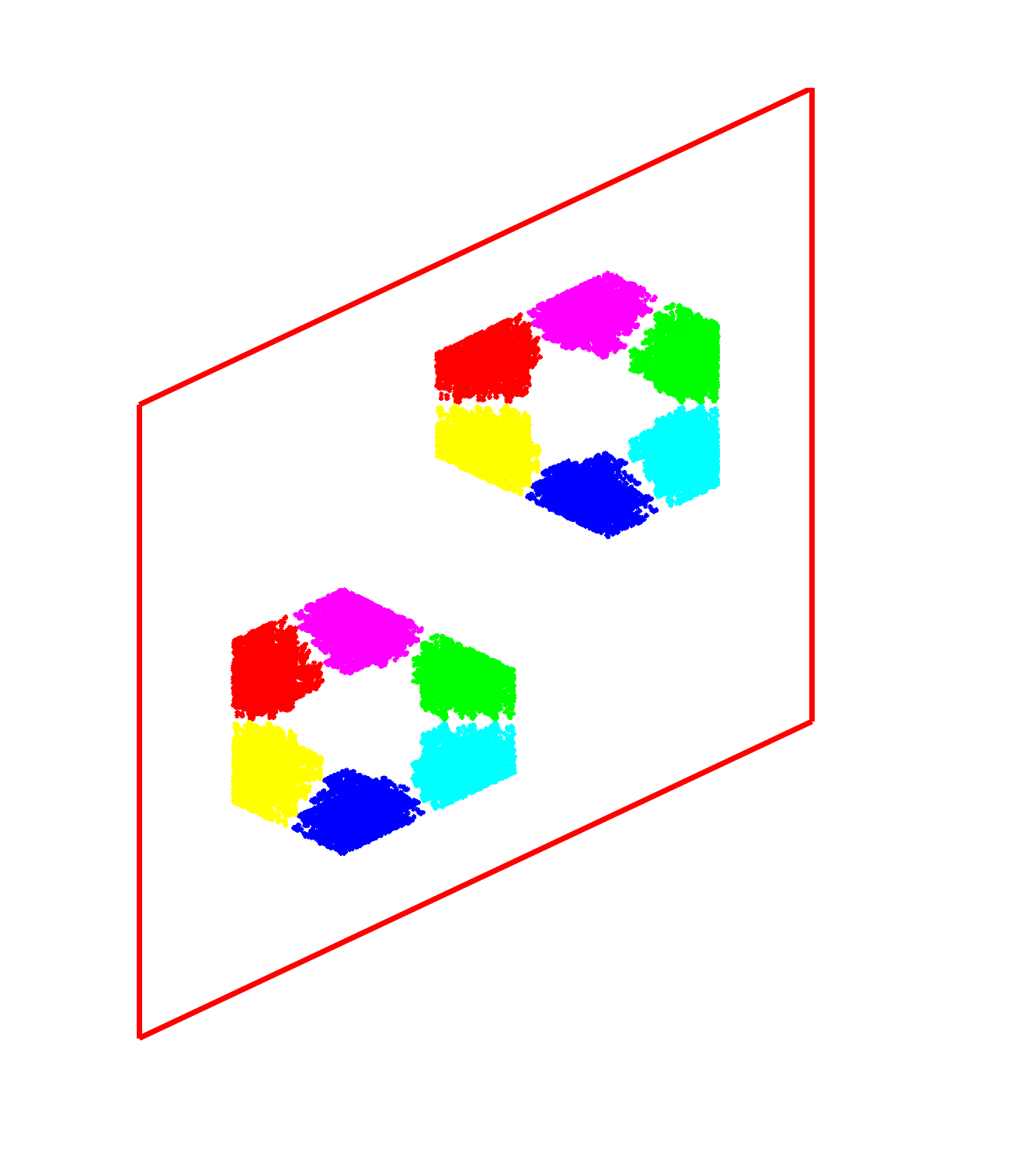}
                          \caption{$\varepsilon=0.42$, Simulation of $G_{\varepsilon,3}^2$.} \label{figb1}
                  \end{subfigure}
                  \begin{subfigure}[b]{0.34\textwidth}
                          \centering
                          \includegraphics[width=\textwidth, height=0.22\textheight]{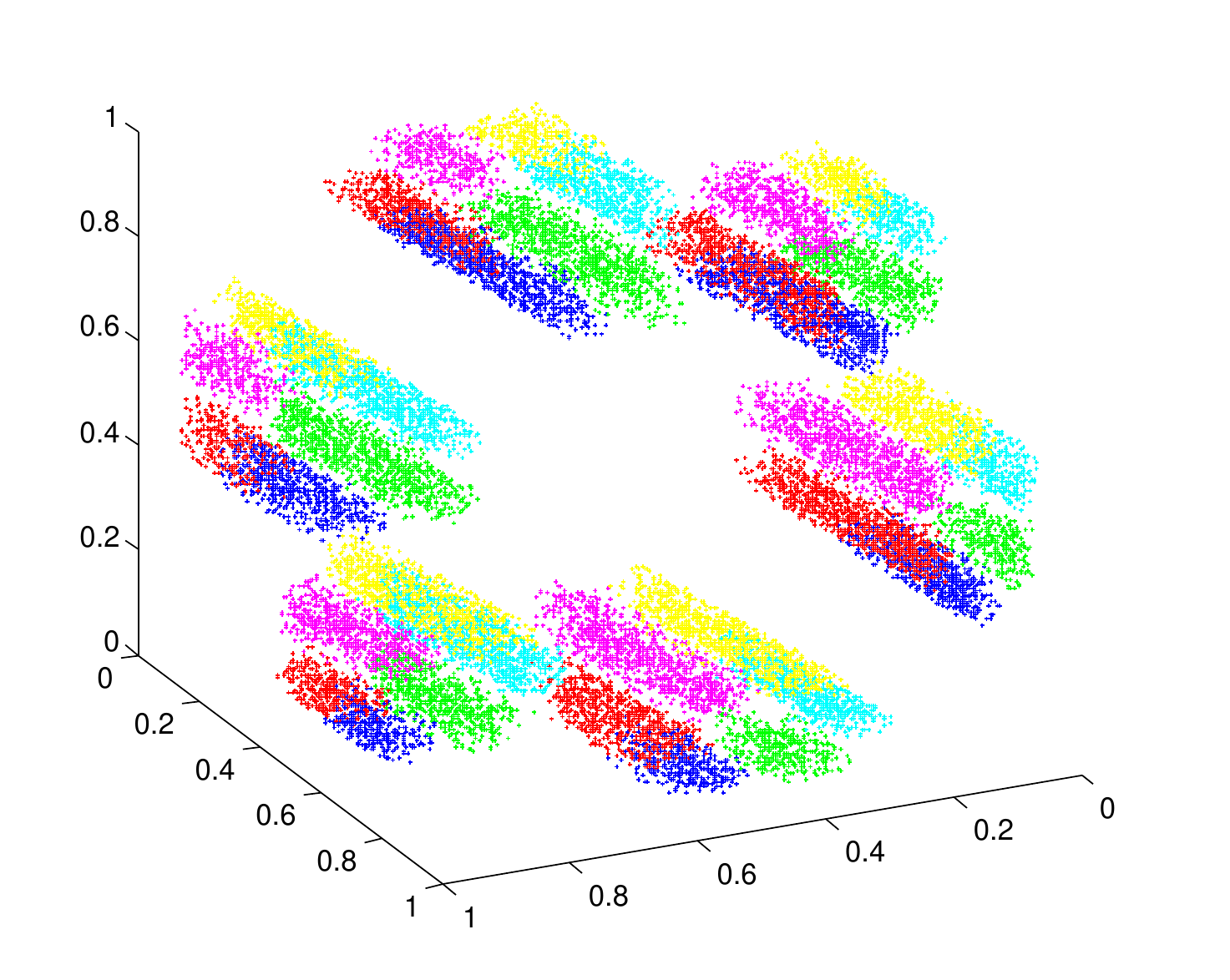}
                          \caption{$\varepsilon=0.42$, Simulation of $F_{\varepsilon,3}$.} \label{figb2}
                  \end{subfigure}
                  \begin{subfigure}[b]{0.34\textwidth}
                          \centering
                          \includegraphics[width=\textwidth, height=0.22\textheight]{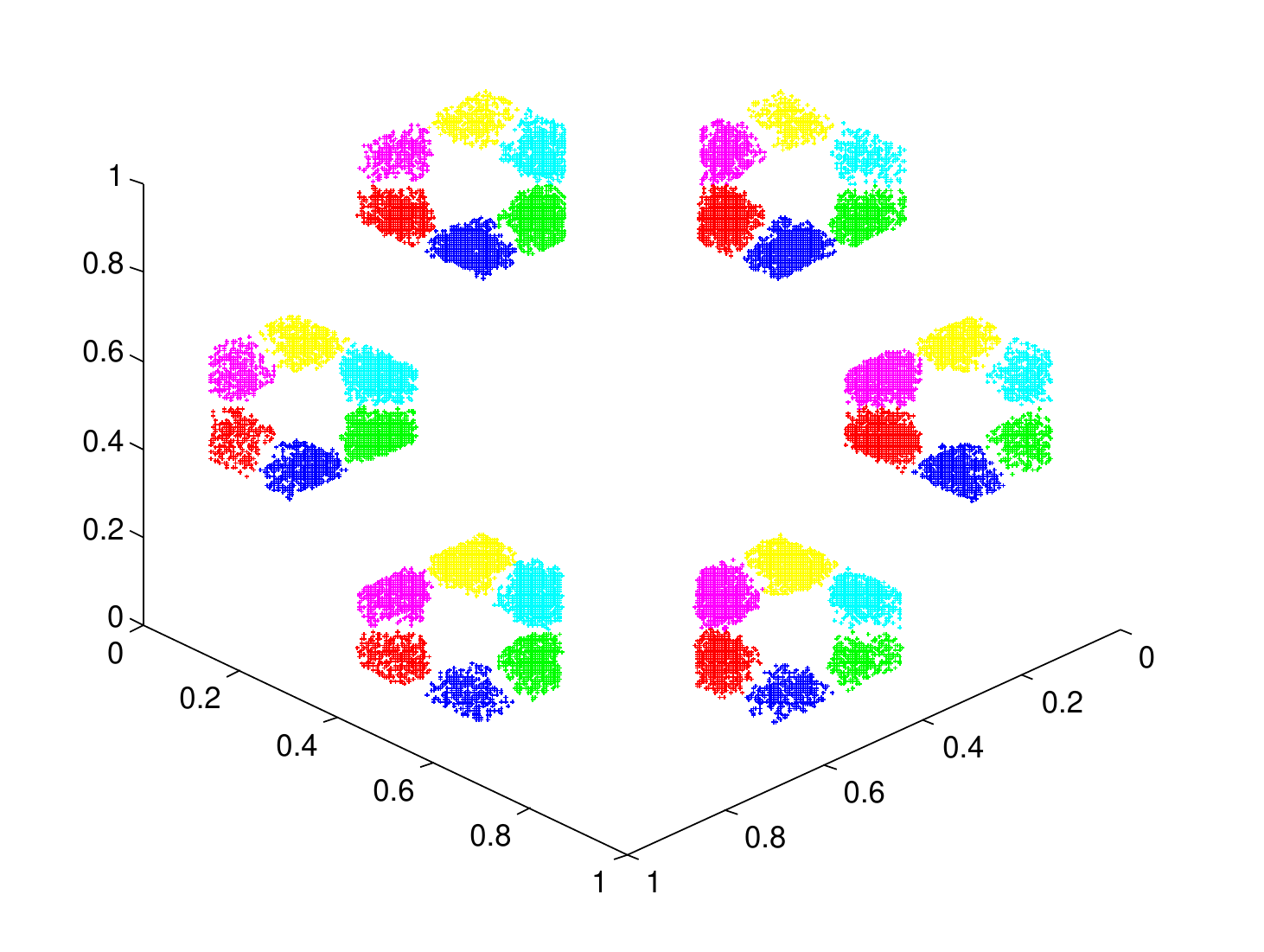}
                          \caption{$\varepsilon=0.42$, Simulation of $F_{\varepsilon,3}$.} \label{figb3}
                  \end{subfigure}
          \caption{Simulations of the systems governed by $F_{\varepsilon,3}$ and $G_{\varepsilon,3}^2$} \label{nagy}
  \end{figure}

  On figure \ref{nagy} we plotted simulations of the two dimensional component of the factor map $G_{\varepsilon,3}^2$ and the original map $F_{\varepsilon,3}$. On figures \ref{figc1}, \ref{figa1}, \ref{figb1} we plotted the last 2000 elements of the 3000 long trajectory of six points with respect to the map $G_{\varepsilon,3}^2$. On figures \ref{figc2}, \ref{figa2}, \ref{figb2} and figures \ref{figc3}, \ref{figa3}, \ref{figb3} we plotted the last 9000 elements of the 10000 long trajectory of six points with respect to the map $F_{\varepsilon,3}$. On figures \ref{figc3}, \ref{figa3}, \ref{figb3} we are looking at the unit cube in the direction of the coordinate $w=x+y+z$ and on figures \ref{figc2}, \ref{figa2}, \ref{figb2} from a slightly different angle. One can see that if $\varepsilon \geq \frac{4-\sqrt{10}}{2} \approx 0.419$, then six invariant components appear both in the factor and the original map.

  It is useful to interpret our results in terms of the original coupled system, that is, the evolution of the position of three sites on $\mathbb{T}$. Let the position of these three sites be $x,y,z \in \mathbb{T}$.  Let $d$ be the quasimetric on $\mathbb{T}$ such that $d(x,y)$ is the length of the counterclockwise arc from $x$ to $y$. Remember that the coordinates $u,v$ of the system $G^2_{\varepsilon,3}$ corresponded to $x-y$ mod 1 and $y-z$ mod 1, respectively. A union of two hexagons, containing the attractor of this system (when $\varepsilon < \frac{1}{2}$,) was plotted on figure \ref{figd3}. The three green lines on the figure (which separate invariant components when  $\frac{4-\sqrt{10}}{2} \leq \varepsilon$) correspond to positions in our original system for which $d(x,y)=d(y,z)$, $d(x,y)=d(z,x)$ and $d(y,z)=d(z,x)$. Hence the invariant components correspond to the following types of states:
   \begin{align*}
   I: & \quad \{x,y,z \in \mathbb{T}: d(x,y) < d(z,x) < d(y,z)\}, \\
   II: & \quad\{x,y,z \in \mathbb{T}: d(x,y) < d(y,z) < d(z,x)\}, \\
   III: &\quad \{x,y,z \in \mathbb{T}: d(y,z) < d(x,y) < d(z,x)\}, \\
   IV: &\quad \{x,y,z \in \mathbb{T}: d(y,z) < d(z,x) < d(x,y)\}, \\
   V: &\quad \{x,y,z \in \mathbb{T}: d(z,x) < d(y,z) < d(x,y)\}, \\
   VI: &\quad \{x,y,z \in \mathbb{T}: d(z,x) < d(x,y) < d(y,z)\}.
   \end{align*}
   Each invariant set $I-VI$ has a component in each hexagon, to be denoted as $Ia, \dots, VIa$ and $Ib, \dots,VIb$. Each component will have a part which stays in the same hexagon (to this correspond states where $\frac{1}{2} < d(x,y), \frac{1}{2} < d(y,z)$ or $\frac{1}{2} < d(z,x)$) - this means that the order of $x,y,z$ on $\mathbb{T}$ will stay the same. There will also be a part which will have an image in the other hexagon - this means that the order of $x,y,z$ will be reversed on $\mathbb{T}$. So if the three sites can fit on an arc with length shorter than $\frac{1}{2}$, their order will stay the same, otherwise it will be reversed.

By the formula which gives the boundary of the hexagons (see figure \ref{figd2}) one can conclude that the points in the Milnor attractor correspond to states for which $\frac{\varepsilon}{3} \leq d(x,y), d(y,z), d(z,x)$. So the long time behaviour of the system is that the sites can get only as close as a third of the coupling strength.

  By interchanging the role of $x,y,z$ we can see that the the domains $I,III,V$ map onto each other, the same is true for domains $II,IV,VI$. Hence in terms of the original
 system with identical sites, only two invariant sets of states exist.

Recall the definitions of the points $A$, $B$ and $C$ from Discussion~\ref{d3}. Direct inspection shows that the components can be described as given there. For instance, in configurations of $Ia$ the three site locations $x, y$ and $z$ follow each other in a counterclockwise order on the circle, $d(x,y)$, $d(y,z)$ and $d(z,x)$ measure the length of the corresponding arcs, of which $d(x,y)$ is the shortest and $d(z,x)$ is the longest; the three quantities add up to 1. If the locations of $y$ and $z$ are exchanged, a configuration in $Ib$ is obtained, for which $d(x,y)$, $d(y,z)$ and $d(z,x)$ correspond to unions of two arcs, and hence add up to $2$, yet their relation defining $I$ persists. Similar description applies to all other components.

   Let us comment finally on the case $\frac{1}{2} < \varepsilon < 1$, when two limit behaviours are possible. The first happens when the initial state of the system corresponds to a point outside the hexagons pictured on figure \ref{figd2} (this roughly means, that the sites are close, in most cases they can fit on an arc with length less than $\frac{1}{2}$). In this case the sites synchronize: they converge to a limit state evolving according to the doubling map. Otherwise they asymptotically acquire an evenly placed position on $\mathbb{T}$.

\section{Continuum of sites} \label{4}

\subsection{The case of small $\varepsilon$: weak interaction}

In this section we study the case when the initial measure is absolutely continuous with respect to the Lebesgue measure on $\mathbb{T}$, which we shall denote by $\lambda$. Let the density function of the initial measure be of smoothness $ C^1$. Denoting the total variation on $\mathbb{T}$ by  $|\cdot|_{TV}$, we further assume that $|f|_{TV}=\delta$ for some $0 \leq \delta$.

Let $\text{d}\mu=f\text{ d}\lambda$. It will suffice to index the dynamics by the density $f$.
\begin{equation} \label{master0}
F_{f}(x)=2x+2\varepsilon \int_0^1g(y-x)f(y)\text{ d}y \qquad \text{ mod 1}.
\end{equation}
First notice that if $f=1$ (that is $\mu=\lambda$), $F_{f}(x)=2x$ mod 1. We can calculate the pushforward density using the Perron-Frobenius operator.
\[
\mathcal{L}_f f(x)=\frac{1}{2}\left(f\left(\frac{x}{2}\right)+f\left(\frac{x+1}{2}\right) \right)=\frac{1}{2}(1+1)=1.
\]
Hence
\[
(F_{1})_{\ast}\lambda=\lambda.
\]
This means that uniformly distributed sites will stay uniformly distributed.

Let us transform formula \eqref{master0} into a more convenient form.  Observe that $g(y-x)=y-x-n(y-x)$ where $n(y-x)$ is an integer depending on $y-x$. More precisely,
\[
n(y-x)=\begin{cases}
-1 & \text{ if } \quad -\frac{3}{2} < y-x < -\frac{1}{2} \\
0 & \text{ if } \quad  -\frac{1}{2} < y-x < \frac{1}{2} \\
1 & \text{ if } \quad \frac{1}{2} < y-x < \frac{3}{2}
\end{cases}
\]
so
\[
n(y-x)=\begin{cases}
-1 & \text{ if } \quad 0 < y < x-\frac{1}{2} \\
0 & \text{ if }  \quad x-\frac{1}{2} < y < x+\frac{1}{2} \\
1 & \text{ if } \quad x+\frac{1}{2} < y < 1
\end{cases}
\]
Using this,
\begin{align*}
F_{f}(x)&=2x+2\varepsilon \left(\int_0^1yf(y)\text{ d}y-x\int_0^1f(y)\text{ d}y-\int_0^1n(y-x)f(y)\text{ d}y \right) \hspace{3.35cm} \text{ mod} 1, \\
F_{f}(x)&=2(1-\varepsilon)x+2\varepsilon \int_0^1yf(y)\text{ d}y-2\varepsilon\int_0^{x-\frac{1}{2}}-f(y)\chi_{\left[x > \frac{1}{2}\right]}\text{ d}y-2\varepsilon\int_{x+\frac{1}{2}}^1f(y)\chi_{\left[ x \leq \frac{1}{2} \right]} \qquad \text{ mod} 1.
\end{align*}
Finally, we get
\begin{equation} \label{master1}
F_{f}(x)=
\begin{cases}
2(1-\varepsilon)x+2\varepsilon \left( \int_0^1 yf(y)\text{ d}y-\int_{x+\frac{1}{2}}^1 f(y)\text{ d}y\right) \hspace{0.15cm} \text{ mod } 1 & \text{ if } x \leq \frac{1}{2} \\
2(1-\varepsilon)x+2\varepsilon \left( \int_0^1 yf(y)\text{ d}y+\int_0^{x-\frac{1}{2}}f(y)\text{ d}y\right) \text{ mod } 1 & \text{ if } x > \frac{1}{2}
\end{cases}
\end{equation}

This function is continuous on the circle:  note that $F_{f}(0)=F_{f}(1)$ mod 1, since $F_{f}(1)-F_{f}(0)=2-2\varepsilon+2\varepsilon \int_0^1f(y)\text{ d}y=2$. This also shows that $F_{f}$ is a degree two covering map of the circle. The derivative
\[
F'_{f}(x)=
\begin{cases}
2\left(1-\varepsilon+\varepsilon f\left(x+\frac{1}{2}\right)\right)& \text{if } 0 < x \leq \frac{1}{2}, \\
2\left(1-\varepsilon+\varepsilon f\left(x-\frac{1}{2}\right)\right)& \text{if } \frac{1}{2} < x \leq 1,
\end{cases}
\]
 is also continuous. Note that $F'_{f} > 0$, hence $F_{f}$ is monotone increasing. The second derivative is
\[
 F''_{f}(x)=
 \begin{cases}
 2\varepsilon f'\left(x+\frac{1}{2}\right) & \text{if } 0 < x \leq \frac{1}{2}, \\
 2\varepsilon f'\left(x-\frac{1}{2}\right) & \text{if } \frac{1}{2} < x \leq 1.
 \end{cases}
\]
  Summing up our observations, the graph of $F_{f}$ has two continuous, increasing branches, which map their domain onto the whole circle. We are going to denote the inverse of these branches by $y_1$ and $y_2$.

Now we move on to prove Theorem \ref{t4}. In this theorem we stated that if we push forward our initial measure with the dynamics $F_{f}$, the resulting measure will have a $C^1$ density function with total variation less than a fraction of $\delta$. Iterating this, we get that the density will tend to a density function with total variation zero (this is necessarily the constant density 1), so the distribution tends to uniform in $BV$ -- if $\varepsilon$ is sufficiently small with respect to $\delta$.

\begin{proof}[Proof of Theorem \ref{t4}.]

Let $f_0=1+g_{\delta}$. Note that since $f_0$ is a density, the integral of $g_{\delta}$ is zero, hence the bound on the total variation implies $\sup |g_{\delta}| < \delta$. Substitutig $f_0=1+g_{\delta}$ into \eqref{master1}, we see that the dynamics takes the form
\begin{align*}
& F_{f_0}(x)=
\begin{cases}
2(1-\varepsilon)x+2\varepsilon \left( \int_0^1 y(1+g_{\delta}(y))\text{ d}y-\int_{x+\frac{1}{2}}^1 (1+g_{\delta}(y))\text{ d}y\right) \hspace{0.15cm} \text{ mod } 1 & \text{ if } x \leq \frac{1}{2} \\
2(1-\varepsilon)x+2\varepsilon \left( \int_0^1 y(1+g_{\delta}(y))\text{ d}y+\int_0^{x-\frac{1}{2}}(1+g_{\delta}(y))\text{ d}y\right) \text{ mod } 1 & \text{ if } x > \frac{1}{2}
\end{cases} \\
& F_{f_0}(x)=
\begin{cases}
2x-2\varepsilon x+\varepsilon-2\varepsilon\left(\frac{1}{2}-x\right)+2\varepsilon \left( \int_0^1 yg_{\delta}(y)\text{ d}y-\int_{x+\frac{1}{2}}^1 g_{\delta}(y)\text{ d}y\right) \hspace{0.15cm} \text{ mod } 1 & \text{ if } x \leq \frac{1}{2} \\
2x-2\varepsilon x+\varepsilon+2\varepsilon\left(x-\frac{1}{2}\right)+2\varepsilon \left( \int_0^1 yg_{\delta}(y)\text{ d}y+\int_0^{x-\frac{1}{2}}g_{\delta}(y)\text{ d}y\right) \text{ mod } 1 & \text{ if } x > \frac{1}{2}
\end{cases}
\end{align*}
Hence
\begin{equation} \label{master2}
F_{f_0}(x)=
\begin{cases}
2x+2\varepsilon \left( \int_0^1 yg_{\delta}(y)\text{ d}y-\int_{x+\frac{1}{2}}^1 g_{\delta}(y)\text{ d}y\right) \hspace{0.15cm} \text{ mod } 1 & \text{ if } x \leq \frac{1}{2} \\
2x+2\varepsilon \left( \int_0^1 yg_{\delta}(y)\text{ d}y+\int_0^{x-\frac{1}{2}} g_{\delta}(y)\text{ d}y\right) \text{ mod } 1 & \text{ if } x > \frac{1}{2}
\end{cases}
\end{equation}
Also note that
\[
F'_{f_0}(x)=
\begin{cases}
2\left(1+\varepsilon g_{\delta}\left(x+\frac{1}{2}\right)\right)& \text{if } 0 < x \leq \frac{1}{2}, \\
2\left(1+\varepsilon g_{\delta}\left(x-\frac{1}{2}\right)\right)& \text{if } \frac{1}{2} < x \leq 1,
\end{cases}
\]
and
\[
 F''_{f_0}(x)=
 \begin{cases}
 2\varepsilon g'_{\delta}\left(x+\frac{1}{2}\right) & \text{if } 0 < x \leq \frac{1}{2}, \\
 2\varepsilon g'_{\delta}\left(x-\frac{1}{2}\right) & \text{if } \frac{1}{2} < x \leq 1.
 \end{cases}
\]

The density function of the pushforward measure can be calculated by applying the Perron-Frobenius operator associated with $F_{f_0}$ to the density $f_0$. Let us use the notation $\mathcal{L}_{F_{f_0}}=\mathcal{L}_{f_0}$. The operator $\mathcal{L}_{f_0}$ is defined as
\[
\mathcal{L}_{f_0} f_0(x)=\sum_{y \in F_{f_0}^{-1}(x)}\frac{f_0(y)}{F_{f_0}'(y)}
\]
We note that the pushforward density $f_1$ always has the same smoothness as $f_0$, because both the inverse branches and the derivative of the dynamics has the same smoothness as $f_0$.

The total variation can be estimated in the following way:
\begin{align*}
&|\mathcal{L}_{f_0} f_0|_{TV}=\int_0^1 \left|\frac{d}{dx}\sum_{y \in F_{f_0}^{-1}(x)}\frac{f_0(y)}{F_{f_0}'(y)}\right|\text{ d}x \\
&=\int_0^1\left| \frac{f_0'(y_1(x))y_1'(x)}{F_{f_0}'(y_1(x))}-\frac{f_0(y)F_{f_0}''(y_1(x))y_1'(x)}{(F_{f_0}'(y_1(x)))^2}+\frac{f_0'(y_2(x))y_2'(x)}{F_{f_0}'(y_2(x))}-\frac{f_0(y)F_{f_0}''(y_2(x))y_2'(x)}{(F_{f_0}'(y_2(x)))^2}\right|\text{ d}x \\
&\leq \int_0^1\left| \frac{f_0'(y_1(x))y_1'(x)}{F_{f_0}'(y_1(x))}\right|\text{ d}x+\int_0^1\left|\frac{f_0(y)F_{f_0}''(y_1(x))y_1'(x)}{(F_{f_0}'(y_1(x)))^2}\right|\text{ d}x+\int_0^1\left|\frac{f_0'(y_2(x))y_2'(x)}{F_{f_0}'(y_2(x))}\right|\text{ d}x \\
&+\int_0^1\left|\frac{f_0(y)F_{f_0}''(y_2(x))y_2'(x)}{(F_{f_0}'(y_2(x)))^2}\right|\text{ d}x=I+II+III+IV.
\end{align*}
Since $-\delta \leq |g_{\delta}| \leq \delta$, we get $\inf|F'_{f_0}|=\inf\left(2\left(1+\varepsilon g_{\delta}\left(x+\frac{1}{2}\right)\right)\right) \geq 2(1-\varepsilon \delta)$ (this implies $\sup \frac{1}{|F'_{f_0}|} \leq \frac{1}{2(1-\varepsilon \delta)}$) and $\sup|f_0| < 1+\delta$. But the most important observation is that in the distortion terms $II$ and $IV$, the derivative of $f_0$ appears, hence in these terms $g_{\delta}$ appears multiplied by $\varepsilon$. Using this we get the following upper bound:
\begin{align*}
&|\mathcal{L}_{f_0} f_0|_{TV} \leq \frac{1}{2(1-\varepsilon \delta)}\int_0^1\left|g_{\delta}'\left(y_1(x)\pm \frac{1}{2}\right)y_1'(x)\right|\text{ d}x+\frac{2\varepsilon (1+\delta)}{(2(1-\varepsilon \delta))^2}\int_0^1\left|g_{\delta}'\left(y_1(x)\pm \frac{1}{2}\right)y_1'(x)\right|\text{ d}x\\
&+\frac{1}{2(1-\varepsilon \delta)}\int_0^1\left|g_{\delta}'\left(y_2(x)\pm \frac{1}{2}\right)y_2'(x)\right|\text{ d}x
+\frac{2\varepsilon (1+\delta)}{(2(1-\varepsilon \delta))^2}\int_0^1\left|g_{\delta}'\left(y_2(x)\pm \frac{1}{2}\right)y_2'(x)\right|\text{ d}x \\ &\leq \frac{1}{2(1-\varepsilon \delta)}\int_0^1|g_{\delta}'(t)|\text{ d}t
+\frac{2\varepsilon (1+\delta)}{(2(1-\varepsilon \delta))^2}\int_0^1|g_{\delta}'(t)|\text{ d}t=\frac{1+\varepsilon}{2(1-\varepsilon \delta)^2}\int_0^1|g_{\delta}'(t)|\text{ d}t \\
&=\frac{1+\varepsilon}{2(1-\varepsilon \delta)^2}|f_0|_{TV}
\end{align*}
Observe that $c=\frac{1+\varepsilon}{2(1-\varepsilon \delta)^2} < 1$ if $\varepsilon < \frac{1}{1+4\delta}$.
\end{proof}
 It follows from the above calculation, that given $\delta$, for sufficiently weak interaction, the density function will converge to a density with zero total variation, hence the distribution of the sites will converge to uniform.

 \subsection{The case of large $\varepsilon$: strong interaction} \label{strongint}

  In this section we are going to introduce a class of initial densities, for which synchronization occurs when the coupling is sufficiently strong -- throughout this section, we are going to assume that $\frac{1}{2} < \varepsilon < 1$.

  \begin{figure}[h!]
   \centering
   \begin{tikzpicture}[scale=2]
         \draw (0,0) -- (2,0);
         \draw[->] (0,-0.1) -- (0,4);
         \draw[very thick, blue] (0,0) -- (0.7,0);
         \draw[very thick, blue] (1.7,0) -- (2,0);
         \draw[very thick, blue] (0.7,0) .. controls (1.2,0.2) and (1.1,3.5) .. (1.2,1.8);
         \draw[very thick, blue] (1.2,1.8) .. controls (1.3,1) and (1.23,3.4) .. (1.3,3.5);
         \draw[very thick, blue] (1.7,0) .. controls (1.2,0.2) and (1.4,3.4) .. (1.3,3.5);
  %       \draw[very thick, blue] (0.35,0) .. controls (0,0.4) and (0.2,1.8) .. (0,3.5);
         \foreach \x/\xtext in {1.7/b_2, 0.7/b_1,2/1}
             \draw[shift={(\x,0)}] (0pt,2pt) -- (0pt,-2pt) node[below] {$\xtext$};
       \end{tikzpicture}
       \caption{The type of density considered: the support is contained in an interval shorter than $\frac{1}{2}$.} \label{smallsup}
       \end{figure}
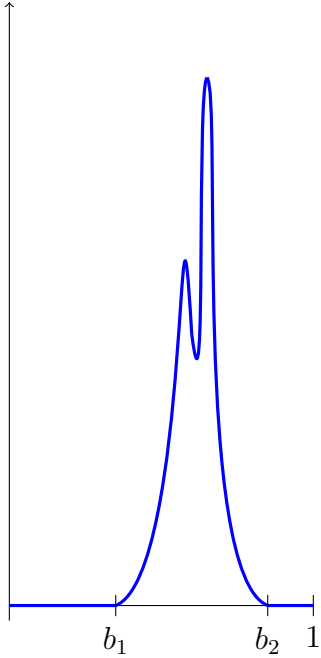

  Let $f_0 \in C(\mathbb{T})$, and let its support be contained in an interval on $\mathbb{T}$ with length less than $\frac{1}{2}$, see figure \ref{smallsup}. We are going to study the case $0 \leq b_1 < \frac{1}{2} < b_2 \leq 1$ in detail, other cases (for example $0 < b_2 < b_1 < 1$) can be handled quite similarly.

 We are first going to prove Theorem \ref{t5}, stating that such densities will converge to a point mass evolving according to the doubling map, when strong interaction is considered.

 \begin{proof}[Proof of Theorem \ref{t5}.]
 Using formula \eqref{master1}, the dynamics defined by such a density is
 \[
 F_{f_0}(x)=\begin{cases}
 2(1-\varepsilon)x+2\varepsilon\left(M-\int_{x+\frac{1}{2}}^1f_0(y)\text{ d}y \right) & \text{ if } 0 \leq x < b_2-\frac{1}{2} \hspace{2cm} \text{ mod } 1\\
 2(1-\varepsilon)x+2\varepsilon M & \text{ if } b_2-\frac{1}{2} \leq x < b_1+\frac{1}{2} \hspace{1.11cm} \text{ mod } 1 \\
 2(1-\varepsilon)x+2\varepsilon\left(M+\int^{x-\frac{1}{2}}_0f_0(y)\text{ d}y \right) & \text{ if } b_1+\frac{1}{2} \leq x < 1 \hspace{2cm} \text{ mod } 1
 \end{cases}
 \]
 because now the center of mass is $M=\int_0^1 yf_0(y)\text{ d}y$. For an illustration, see figure \ref{dyn2}.
 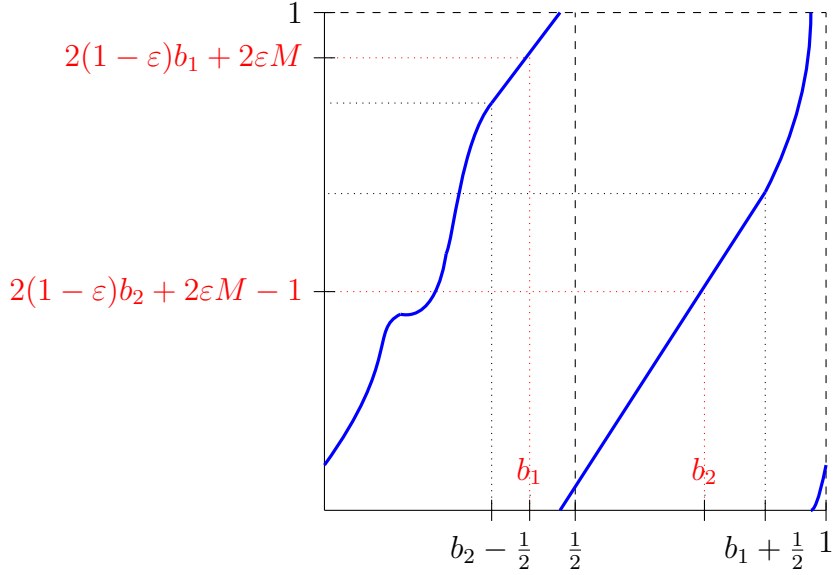
\begin{figure}[h!]
  \centering
  \begin{tikzpicture}[scale=2]
        \draw (0,0) -- (3.3,0);
        \draw (0,0) -- (0,3.3);
        \draw[dashed] (3.3,0) -- (3.3,3.3);
        \draw[dashed] (0,3.3) -- (3.3,3.3);
        \draw[dashed] (1.65,0) -- (1.65,3.3);
        \draw[dotted] (1.1,2.7) -- (0,2.7);
        \draw[dotted] (2.9,2.1) -- (0,2.1);
        \draw[dotted] (1.1,0) -- (1.1,2.7);
        \draw[dotted] (2.9,0) -- (2.9,2.1);
        \draw[dotted,red] (1.35,0) -- (1.35,3) -- (0,3);
        \draw[dotted,red] (2.5,0) -- (2.5,1.45) -- (0,1.45);

        \draw[very thick, blue] (1.1,2.7) -- (1.55,3.3);
        \draw[very thick, blue] (1.55,0) -- (2.9,2.11);
        \draw[very thick, blue] (0,0.3) .. controls (0.5,1) and (0.3,1.2) .. (0.5,1.3);
        \draw[very thick, blue] (0.5,1.3) .. controls (0.75,1.25) and (0.8,1.7) .. (0.8,1.7);
        \draw[very thick, blue] (0.8,1.7) .. controls (0.85,1.8) and (0.9,2.45) .. (1.1,2.7);
 %       \draw[very thick, blue] (0.7,2) .. controls (0.7,2.3) and (1.1,2.67) .. (1.1,2.7);
        \draw[very thick, blue] (2.9,2.11) .. controls (3.1,2.5) and (3.2,2.9) .. (3.2,3.3);
        \draw[very thick, blue] (3.2,0) .. controls (3.22,0) and (3.25,0.05) .. (3.3,0.3);
        \foreach \x/\xtext in {1.65/\frac{1}{2},2.9/b_1+\frac{1}{2},1.1/b_2-\frac{1}{2},3.3/1}
            \draw[shift={(\x,0)}] (0pt,2pt) -- (0pt,-2pt) node[below] {$\xtext$};
        \foreach \x/\xtext in {2.5/\textcolor{red}{b_2},1.35/\textcolor{red}{b_1}}
                   \draw[shift={(\x,0.1)}] (0pt,0pt) -- (0pt,0pt) node[above] {$\xtext$};
        \draw (2.5,-0.07) -- (2.5,0.07);
        \draw (1.35,-0.07) -- (1.35,0.07);
        \foreach \y/\ytext in {3/\textcolor{red}{2(1-\varepsilon)b_1+2\varepsilon M },1.45/\textcolor{red}{2(1-\varepsilon)b_2+2\varepsilon M-1},3.3/1}
                   \draw[shift={(0,\y)}] (2pt,0pt) -- (-2pt,0pt) node[left] {$\ytext$};
      \end{tikzpicture}
      \caption{The dynamics $F_{f_0}$ defined by the density $f_0$} \label{dyn2}
      \end{figure}
 If $b_2-b_1 \leq \frac{1}{2}$, then $[b_1,b_2] \subseteq \left[b_2-\frac{1}{2}, b_1+\frac{1}{2} \right]$, so
 \begin{align*}
 F_{f_0}(b_1)&=2(1-\varepsilon)b_1+2\varepsilon M, \\
 F_{f_0}(b_2)&=2(1-\varepsilon)b_2+2\varepsilon M-1.
 \end{align*}
 This means that the pushforward density is
 \[
 f_1=\mathcal{L}_{f_0} f_0(x)=\begin{cases}
 \frac{f_0\left(\frac{x-2\varepsilon M+1}{2(1-\varepsilon)} \right)}{2(1-\varepsilon)} & \text{ if } 0  \leq x \leq  2(1-\varepsilon)b_2+2\varepsilon M-1 \\
 \frac{f_0\left(\frac{x-2\varepsilon M}{2(1-\varepsilon)} \right)}{2(1-\varepsilon)} & \text{ if } 2(1-\varepsilon)b_1+2\varepsilon M \leq x \leq 1 \\
 0 & \text{ otherwise }
 \end{cases}
 \]
  If we use the notation $b'_1=2(1-\varepsilon)b_1+2\varepsilon M$ and $b'_2=2(1-\varepsilon)b_2+2\varepsilon M-1$ for the boundary of $\supp f_1$, we see that
 \[
 d(b'_1,b'_2)=|b'_2-b'_1|=1-b'_1+b'_2=2(1-\varepsilon)|b_2-b_1|=2(1-\varepsilon)d(b_1,b_2).
 \]
 The supremum of $f$ is now multiplied with a factor of $ \frac{1}{2(1-\varepsilon)}$.

 It remains to prove that the center of mass evolves according to the doubling map.

 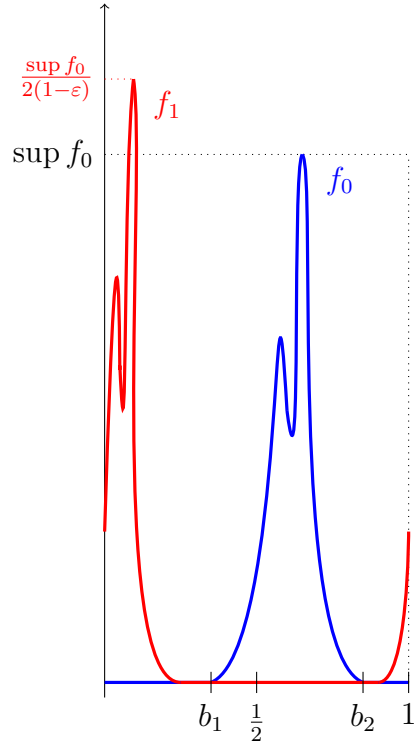
\begin{figure}[h!]
  \centering
  \begin{tikzpicture}[scale=2]
        \draw (0,0) -- (2,0);
        \draw[->] (0,-0.1) -- (0,4.5);
        \draw[dotted] (2,0) -- (2,3.5);
        \draw[dotted,red] (0,4) node[left] {$\frac{\sup f_0}{2(1-\varepsilon)}$} -- (0.2,4);
        \draw[dotted] (0,3.5) node[left] {$\sup f_0$} -- (2,3.5);
        \draw[very thick, blue] (0,0) -- (0.7,0);
        \draw[very thick, blue] (1.7,0) -- (2,0);
        \draw[very thick, blue] (0.7,0) .. controls (1.2,0.2) and (1.1,3.5) .. (1.2,1.8);
        \draw[very thick, blue] (1.2,1.8) .. controls (1.3,1) and (1.23,3.4) .. (1.3,3.5);
        \draw[very thick, blue] (1.7,0) .. controls (1.2,0.2) and (1.4,3.4) .. (1.3,3.5) node [below right] {\hspace{0.15cm}$f_0$};
        \draw[very thick, red,shift={(-1.1,0)}] (1.6,0) -- (2.9,0);
        \draw[very thick, red,shift={(-1.1,0)}] (1.1,1) .. controls (1.2,4.2) and (1.2,1.8) .. (1.2,2.1);
        \draw[very thick, red,shift={(-1.1,0)}] (2.9,0) .. controls (3.1,0) and (3.1,1) .. (3.1,1);
        \draw[very thick, red,shift={(-1.1,0)}] (1.2,2.1) .. controls (1.25,1) and (1.22,3.4) .. (1.29,4);
        \draw[very thick, red, shift={(-1.1,0)}] (1.6,0) .. controls (1.1,0) and (1.38,3.4) .. (1.29,4) node[below right] {\hspace{0.1cm}\textcolor{red}{$f_1$}};
 %       \draw[very thick, blue] (0.35,0) .. controls (0,0.4) and (0.2,1.8) .. (0,3.5);
        \foreach \x/\xtext in {1/\frac{1}{2},1.7/b_2, 0.7/b_1,2/1}
            \draw[shift={(\x,0)}] (0pt,2pt) -- (0pt,-2pt) node[below] {$\xtext$};
      \end{tikzpicture}
      \caption{The density $f_0$ and the pushforward density $f_1$} \label{pushf2}
      \end{figure}

 \begin{align*}
 M(f_1)&=\int_0^{2(1-\varepsilon)b_2+2\varepsilon M-1}(y+1)\frac{f_0\left(\frac{y-2\varepsilon M+1}{2(1-\varepsilon)} \right)}{2(1-\varepsilon)}\text{ d}y+\int_{2(1-\varepsilon)b_1+2\varepsilon M}^1 y \frac{f_0\left(\frac{y-2\varepsilon M}{2(1-\varepsilon)} \right)}{2(1-\varepsilon)} \text{ d}y \\
 &=\int_{\frac{1-\varepsilon M}{2(1-\varepsilon)}}^{b_2}(2(1-\varepsilon)t+2\varepsilon M)f_0(t)\text{ d}t+\int^{\frac{1-\varepsilon M}{2(1-\varepsilon)}}_{b_1}(2(1-\varepsilon)t+2\varepsilon M)f_0(t)\text{ d}t \\
 &=\int_{b_1}^{b_2}(2(1-\varepsilon)t+2\varepsilon M)f_0(t)\text{ d}t \\
 &=2\int_{b_1}^{b_2}tf_0(t)\text{ d}t-2\varepsilon\int_{b_1}^{b_2}tf_0(t)\text{ d}t+2\varepsilon M \int_{b_1}^{b_2}f_0(t)\text{ d}t \\
 &=2\int_{b_1}^{b_2}tf_0(t)\text{ d}t-2\varepsilon M+2\varepsilon M=2M \qquad \text{ mod } 1.
 \end{align*}
 \end{proof}

 Since $2(1-\varepsilon) < 1$ if $\frac{1}{2} < \varepsilon < 1$, we see that when such strong interaction is considered, the density is rescaled in such a way that the length of its support shrinks to a fraction of the original length, while the supremum multiplies by some constant greater than 1. The density is also shifted such that the center of mass evolves according to the doubling map. Iterating this, we see that the density converges to a point mass evolving according to the doubling map.

 We note that when $\varepsilon = \frac{1}{2}$, the density suffers no rescaling, it simply shifts on the circle. If the center of mass is zero, such densities happen to be invariant.

 Let us now look at the case when the support of the density is not contained in an interval shorter than the half of the torus, but an interval of length shorter than $\frac{1}{2}$ supports most of the mass. In Theorem \ref{t6} we stated that in this case the dynamics will shrink the support of the density into an interval of length less than $\frac{1}{2}$ in one step for sufficiently large $\varepsilon$. From then our previous analysis will apply.

 \begin{proof}[Proof of Theorem \ref{t6}.]
 The dynamics takes the form
 \[
 F_{f_0}(x)=\begin{cases}
 2(1-\varepsilon)x+2\varepsilon\left(M-\int_{x+\frac{1}{2}}^1f_0(y)\text{ d}y \right) & \text{ if } 0 \leq x < b_2-\frac{1}{2} \hspace{2cm} \text{ mod } 1 \\
 2(1-\varepsilon)x+2\varepsilon M & \text{ if } b_2-\frac{1}{2} \leq x < b_1+\frac{1}{2} \hspace{1.11cm} \text{ mod } 1 \\
 2(1-\varepsilon)x+2\varepsilon\left(M+\int^{x-\frac{1}{2}}_0f_0(y)\text{ d}y \right) & \text{ if } b_1+\frac{1}{2} \leq x < 1 \hspace{2cm} \text{ mod } 1
 \end{cases}
 \]
 The pushforward density $f_1$ will have a support contained in the interval of $\mathbb{T}$ $[0,b_2'] \cup [b_1',0]$, where
 \begin{align*}
 b_1'&=2(1-\varepsilon)b_1+2\varepsilon\left(M-\int_{b_1+\frac{1}{2}}^1f_0(y)\text{ d}y \right) \hspace{0.43cm} \text{ mod 1}, \\
 b_2'&=2(1-\varepsilon)b_2+2\varepsilon\left(M+\int^{b_2-\frac{1}{2}}_0f_0(y)\text{ d}y \right) \quad \text{mod 1}.
 \end{align*}
 From this,
 \[
 d(b_2',b_1')=|b'_2-b'_1|=2(1-\varepsilon)(b_2-b_1)+2\varepsilon\mathcal{C},
 \]
 where
 \[
 \mathcal{C}=\int_0^{b_2-\frac{1}{2}}f_0(y)\text{ d}y+\int_{b_1+\frac{1}{2}}^1f_0(y)\text{ d}y.
 \]
 We see that $2(1-\varepsilon)(b_2-b_1)+2\varepsilon\mathcal{C} \leq \frac{1}{2}$ holds precisely if $\frac{(b_2-b_1)-\frac{1}{4}}{(b_2-b_1)-\mathcal{C}} \leq \varepsilon$.
 \end{proof}

 We conclude our discussion of the case of absolutely continuous initial measure by noting that in all three cases studied, the time averages $A(T)=\frac{1}{T}\sum_{t=0}^{T-1}\mu_t$ converge to Lebesgue. In the first case, when $\varepsilon$ is small compared to the total variation of the initial density, in fact $\lim_{t \to \infty}\mu_t=\lambda$ holds, since the corresponding densities $f_t$ converge to 1 in BV. In the latter two cases, the limit $\lim_{t \to \infty}\mu_t$ does not exist, but the measures $\mu_t$ converge to a point mass with support evolving according to the doubling map, hence $\lim_{t \to \infty}A(T)=\lambda$ (in typical cases, namely when the center of mass of the initial distribution is a Lebesgue-typical point of the doubling map.)

\section{Concluding remarks} \label{5}

Our main goal in this paper was to contribute to the understanding of synchronization phenomena that arise in globally (mean field) coupled chaotic maps. Specifically,
we investigated diffusively coupled doubling maps. In this setting, and specifically for $N=3$ sites, two distinct values of the coupling parameter were identified in the literature: $\eps=\frac12$  which corresponds to the emergence of contracting directions (\cite{koiller2010coupled}), and a strictly smaller value of $\eps$ which corresponds to the loss of ergodicity (\cite{fernandez2014breaking}). In particular we have addressed two questions: How these dynamical phenomena can be interpreted in terms of the synchronization of the sites? Is there a way to detect analogous  phenomena in the system with a continuum of sites? To this end, we have reformulated the problem in terms of studying the evolution of distributions on the circle. This point of view incorporates the case of finitely many sites (pure point measures) and allows to include the case of a continuum of sites (absolutely continuous distributions).

Our main findings are as follows:
\begin{itemize}
\item A bifurcation arises prior to the emergence of contracting directions even in the case of $N=2$ sites, which corresponds to the loss of mixing (the system remains ergodic throughout the expanding regime).
\item For $N=3$ sites, the ergodic components that arise have a clear interpretation in terms of the relative positions of the sites. In particular, this provides further evidence that the loss of ergodicity corresponds to the breaking of the inversion and the permutation symmetries.
\item For $\eps>\frac12$, when contracting directions are present, the $N=3$ system has two asymptotic states: either the sites take the same position on the circle (attractive synchronization) or the three sites are evenly distributed on the circle (repulsive synchronization). The basins of these asymptotic states can be identified, in particular, for sufficiently concentrated initial distributions the configuration converges to the attractive, while for sufficiently even initial distributions it converges to the repulsive initial state.
\item For absolutely continuous measures modeling a continuum of sites, on the one hand we have proved that an initial distribution that is sufficiently close to the uniform tends, for an appropriately small $\eps$, to the uniform distribution asymptotically. On the other hand, for $\eps>\frac12$, we have identified a class of sufficiently concentrated initial distributions that tend to a point mass evolving chaotically on the circle. In particular, for $\eps>\frac12$ two different limit behaviours can be identified, which both attract a class of -- even vs. concentrated -- initial distributions.
\end{itemize}

As for the the last two items, we can conclude that for $\eps>\frac12$ (the contracting regime) analogous dynamical phenomena can be observed for a continumm of sites and for finitely many sites (more precisely $N=3$) sites. Whether it is possible to detect dynamical phenomena in the system with a continuum of sites that is analogous to the loss of ergodicity  observed for $N=3$, it remains to be seen. In fact, we do not expect phenomena analogous to the breaking of the permutation symmetry as this cannot be interpreted at the level of distributions. The breaking of the inversion symmetry, however, requires further investigation.

To conclude we make one more comment concerning this last point. For a class of measures on the circle, it is natural to consider the support of the distribution. In case of finitely many sites, this is the shortest proper subinterval on which all the sites are located. For absolutely continuous distributions, this is the subinterval (if such an interval exists) on which the density is strictly positive, and outside of which the density vanishes; this makes sense, in particular, for the distributions of Theorem~\ref{t5}. When restricted to the
support, it makes sense to consider the mean (or center of mass) for such a measure, as in Theorem~\ref{t5}.

Now recall from Discussion~\ref{d3} that the loss of ergodicity corresponding to even and odd components can be interpreted as a breaking of the inversion symmetry.
This can be formulated as follows: for odd and even components the center of mass is located to the left and to the right, respectively, of the geometric midpoint of
the supporting interval. It is then possible to investigate the analogous property for a continuum of sites. To this end, we have seen in section~\ref{4} that for
$\eps>\frac12$ the sufficiently concentrated distributions evolve in such a way that the center of mass moves according to the doubling map, while the distribution is simply
rescaled. Hence the property whether the center of mass is to the left or to the right of the geometric midpoint of the support is invariant for the dynamics. On the other hand, for $\eps<\frac12$ it is easy to see that any absolutely continuous initial distribution will be supported on the whole circle after sufficiently many steps of the dynamics.
Thus, the question whether the breaking of inversion symmetry arises in this case, requires different tools and should be the subject of future research.

\section*{Acknowledgements} Stimulating discussions with Bastien Fernandez are thankfully acknowledged. We are grateful
to B\'alint T\'oth for his suggestion to consider the evolution of distributions. We thank the anonymous referees for their valuable comments. This work was
partially supported by Hungarian National Foundation for Scientific Research (NKFIH OTKA)
grant K104745, and Stiftung Aktion \"Osterreich Ungarn (A\"OU) grants 87\"ou6 and 92\"ou6.
\newpage

\bibliographystyle{plain}
\bibliography{references2}
 \nocite{*}

\end{document}